\documentclass[11pt,a4paper,notitlepage,oneside]{article}

\usepackage[utf8x]{inputenc} 
\usepackage{ucs} 
\usepackage[english]{babel}
\usepackage{amsmath}
\usepackage{amsthm}
\usepackage{amsfonts}
\usepackage{mathtools}
\usepackage{amssymb} 
\usepackage{nicefrac}
\usepackage[percent]{overpic}
\usepackage{hyperref}
\hypersetup{
	colorlinks=true,  
	linkcolor=black,    
	urlcolor=red      
}
\usepackage{authblk}

\usepackage{enumerate} 
\usepackage[short]{optidef}
\usepackage{array}
\usepackage{ulem}
\usepackage{tikz}
\tikzstyle{vertex}=[circle,fill=black!25,minimum size=20pt,inner sep=0pt,draw]

\usepackage{booktabs}
\usepackage{longtable}
\usepackage[textsize = tiny]{todonotes}  

\theoremstyle{definition}
\newtheorem{definition}{Definition}
\newtheorem{observation}[definition]{Observation}
\newtheorem{remark}[definition]{Remark}
\newtheorem{proposition}[definition]{Proposition}

\newtheorem{assumption}{Assumption}
\newtheorem{problem}[definition]{Problem}

\theoremstyle{plain}
\newtheorem{theorem}[definition]{Theorem}
\newtheorem{corollary}[definition]{Corollary}
\newtheorem{lemma}[definition]{Lemma}

\newcommand{\nonRobust}{$q$-Multiset Multicover}
\newcommand{\nonRobustShort}{$q$-MSMC}
\newcommand{\nonRobustOpt}{Min-$q$-Multiset Multicover}
\newcommand{\nonRobustOptShort}{Min-$q$-MSMC}

\newcommand{\RobustShort}{Robust $q$-MSMC}
\newcommand{\RobustOpt}{Robust Min-$q$-Multiset Multicover}
\newcommand{\RobustOptShort}{Robust Min-$q$-MSMC}

\newcommand{\RobustOptShortT}{Robust Min-$3$-MSMC}

\newcommand{\bx}{\bar{x}}
\newcommand{\Z}{\mathbb{Z}}

\date{\today}

\begin{document}
\normalem 

\title{Robust Multicovers with Budgeted Uncertainty}
\author{Sven O.~Krumke}
\author[1]{Eva Schmidt}
\author[1]{Manuel Streicher}
\affil[1]{Department of Mathematics, Technische Universit\"at Kaiserslautern, Germany}

\setcounter{Maxaffil}{0}
\renewcommand\Affilfont{\itshape\small}

\maketitle

\begin{abstract}
	The Min-$q$-Multiset Multicover problem presented in this
	paper is a special version of the Multiset Multicover problem. For a
	fixed positive integer~$q$, we are given a finite ground set~$J$, an
	integral demand for each element in~$J$ and a collection
	of subsets of~$J$. The task is to choose
	sets of the collection (multiple choices are allowed) such that
	each element in~$J$ is covered at least as many times as specified by the demand
	of the element. In contrast to Multiset Multicover, in Min-$q$-Multiset Multicover each of
	the chosen subsets may only cover up to~$q$ of its elements with multiple choices being allowed.
	
	Our main focus is a robust version of
	Min-$q$-Multiset Multi\-cover, called Robust Min-$q$-Multiset Multicover, in which the demand of
	each element in $J$ may vary in a given interval with an
	additional budget constraint bounding the sum of the demands.
	Again, the task is to find a selection of subsets
		which is feasible for all admissible demands.
	
	We show that the non-robust version is NP-complete for $q$ greater than two, whereas the robust version is strongly NP-hard for any positive $q$.
	Furthermore, we present two solution approaches based on constraint generation and investigate
	the corresponding separation problems.
	
	We present computational results using randomly generated
	instances as well as instances emerging from the problem of
	locating emergency doctors.
\end{abstract}
\let\thefootnote\relax\footnotetext{\textsuperscript{\textcopyright}2018. This manuscript version is made available under the CC-BY-NC-ND 4.0 license, \url{http://creativecommons.org/licenses/by-nc-nd/4.0/}}

\section{Introduction}
\label{sec: Introduction}
Covering problems arise in many real world applications.  Therefore, there has been a lot of research regarding this area of optimization.  
In the classical \emph{Set Cover} problem we are given a set $U$, a collection of subsets $\mathcal{S}\subseteq 2^U$ and a positive integer~$k$.  
The decision version asks for the existence of a subcollection $\mathcal{S}'\subseteq\mathcal{S}$ of size at most $k$, such that, for all $u\in U$, there is some $S\in\mathcal{S}'$ with $u\in S$.  
This problem is well-known to be strongly NP-complete, see~\cite{Garey}.  
In the \emph{Set Multicover} problem each element $u\in U$ is given a demand $d_u\in\mathbb{Z}_{\geq 0}$ expressing the number of times the element $u$ has to be covered. 
Finally, considering the \emph{Multiset Multicover} problem, the subsets in the collection may be multisets, cf.~\cite{hua}.  
These two variants are clearly generalizations of \emph{Set Cover} and therefore remain strongly NP-complete.  
Furthermore, the problems above remain strongly NP-complete if only subsets of a fixed size $q\geq 3$ are regarded, cf.~\cite{Garey}.

The notion of robustness has gained a lot of attention in operations research. 
The core idea of robust optimization can be summarized as follows: At the time of computation not all data of the instance may be known exactly.  Instead of fixed parameters we are given a \emph{set of scenarios}~$\mathcal{U}$, the \emph{uncertainty set}, where each scenario defines $n$ fixed parameters for some $n\in\mathbb{N}$. 
We assume that any of the scenarios contained in $\mathcal{U}$ may actually occur, but we do not know the \emph{true} scenario in advance. The aim is to find a solution taking into account all scenarios.  
The work was pioneered in~\cite{Soyster:1973} and become a major research area within the optimization community  with~\cite{BenTal+Nemirovsky:1998,BenTal+Nemirovsky:1999,BenTal+Nemirovsky:2000}. A thorough general introduction and overview of robust optimization can be found in~\cite{ROBook}. Furthermore, a recent overview is given by~\cite{GABREL2014471}.

We denote the scenarios by vectors $\xi\in\mathbb{R}^n$ where each entry corresponds to some parameter of the instance. Several methods of defining uncertainty sets have been proposed in current literature, cf.~\cite{BertsimasPriceOfRobustness, Bertsimas2003, kouvelis1996robust, kasperski2008discrete} for a general overview. One arising concept is that of \emph{discrete uncertainty} where the uncertainty sets may only contain a finite number of possible scenarios, cf.~\cite{kouvelis1996robust,Kasperski+Zielinski:2016}. Further, when considering \emph{interval uncertainty} the uncertainty set can be described as
  \begin{align*}
  \mathcal{U}=\left\{\xi\in\mathbb{R}^n \colon \xi_i\in[a_i,b_i], i=1,\dotsc,n\right\},
\end{align*}
for some $a_i,b_i\in\mathbb{R}$, cf.~\cite{kouvelis1996robust}. In this paper, we investigate \emph{discrete budgeted uncertainty} as a combination of the above concepts, i.e.,\
\begin{align*}
\mathcal{U}=\left\{\xi\in\mathbb{R}^n\colon \xi_i\in[a_i,b_i]\cap\mathbb{Z}, i=1,\dotsc,n \text{ and } \sum_{i=1}^n \xi_i\leq\Gamma\right\},
\end{align*}
for some $a_i,b_i,\Gamma\in\mathbb{Z}$, cf.~\cite{BertsimasPriceOfRobustness, Bertsimas2003}. Note that this definition of discrete budgeted uncertainty set differs from other settings using the same expression, e.g. \cite{Chassein,bougeret:hal-01345283,Nasrabadi}. Here we bound the total sum of the uncertainty values. In our paper, the scenarios define parameters of the constraints and do not appear in the objective function. We aim for a solution that fulfills the constraints for all possible scenarios.

In~\cite{Gupta+etal:robust-cover}, the authors studied robust versions of the classical \emph{Set Cover} problem where the possible scenarios are given by all demand-subsets of a certain fixed size.  They provide approximation algorithms for robust two-stage problems:\ Some of the sets may be selected in a first stage at lower cost and in a second stage, after the scenario is known, the remaining sets are chosen.  
Further, an approximation algorithm using the online algorithm for \emph{Set Cover}~\cite{Alon+etal:online-setcover,Buchbinder+Naor:online-primal-dual} within an LP-rounding-based algorithm can be found in~\cite{Feige+etal:robust-cover}.  The robust \emph{Set Cover} problem was also studied from a polyhedral point of view in~\cite{Fischetti:cutting-plane-robust}, whereas new formulations for robust \emph{Set Cover} problems were given in~\cite{Lutter+etal:cover}.


In this paper we introduce for fixed $q\in\mathbb{N}$ the problem \nonRobust, which can be located between \emph{Set Cover by $q$-sets} (cf.~\cite[SP2]{Garey} for the version with $q=3$) and \emph{Multiset Multicover}~\cite{Hua+etal:multicover,Rajagopalan+etal:multicover} as we will see later.  In \nonRobust, the subsets have arbitrary size, but the number of elements they may cover is bounded by $q$.  In fact it is closely related to \emph{Multiset Multicover by $q$-sets}.  However, this relation cannot be generalized to the following robust version of this problem constituting the main focus of this paper.

We investigate a robust version of \nonRobust\ with discrete budgeted uncertainty where the scenarios correspond to demand vectors. 
After analyzing the complexity of all introduced problems, we present different solution approaches based on constraint generation and give computational results for both random instances and instances inspired by a real world problem.

To that end, we discuss an application of \nonRobust\ in the healthcare sector which motivated the study of robust multicovers. In the application we are asked to assign emergency doctors to facilities such that occurring emergencies may be handled in a satisfactory manner. The number of emergencies are uncertain and are represented using the proposed discrete budgeted uncertainty set where the total number of occurring emergencies is budgeted to avoid unrealistic situations. This leads to a multicover problem, where the elements are the regions in which the emergencies occur and the sets correspond to the subsets of regions which can be reached within a guaranteed response time from the facility chosen.  

The article is organized as follows: In Section~\ref{sec: Problem Definition and Classifications}, we introduce \nonRobust, present possible integer pro\-gramming for\-mu\-lations and prove NP-completeness for $q\geq 3$. 
Section~\ref{sec: Robust Version} deals with the robust version of \nonRobust. 
We discuss whether the introduced formulations can be transferred and show that the problem is NP-hard for any $q>0$.
Solution approaches are presented in Section~\ref{sec: Solving the Robust Version}, while Section~\ref{sec: Computational Results} displays the corresponding computational results.


\section{Problem definition and classifications}
\label{sec: Problem Definition and Classifications}
Let $G=(V,E)$ be an undirected graph.
We denote by $N_G(v)$ the
neighborhood of $v\in V$ in~$G$, i.e., the set of all vertices adjacent
to~$v$. 
For a subset $S\subseteq V$, $N_G(S)$ is the set
of all nodes adjacent to some node in~$S$.  
For a directed graph
$G=(V,R)$ we indicate by $N^{+}_G(v)$ the set of successors of $v\in V$,
i.e., the set of vertices~$w$ such that there is a directed arc
from~$v$ to~$w$. 
Analogously, by $N^{-}_G(v)$ we denote the set of
predecessors of $v\in V$.  
If the corresponding graph $G$ is clear
from the context, we omit the subscript~$G$.  
Now, we may
formally define \nonRobust\ for a fixed integer $q\in \Z_{>0}$:

\begin{problem}[\nonRobust~(\nonRobustShort)]
\label{prob: MSMC}
\mbox{}\\
\noindent\textbf{Instance:} 
Finite ground set $J$, weights $d_j\in\mathbb{Z}_{\geq 0}$ for all $j\in J$, a collection of subsets $\mathcal{J}\subseteq 2^J$ and a positive integer
$B\in\mathbb{Z}_{>0}$.

\noindent\textbf{Question:}
Are there integers $x_A\in\mathbb{Z}_{\geq 0}$ for $A\in\mathcal{J}$ with $\sum_{A\in\mathcal{J}}x_A\leq B$, 
such that there exist integers $y_{Aj}\in\mathbb{Z}_{\geq 0}$ for $A\in\mathcal{J}$, $j\in J$ 
satisfying
\begin{align*}
\sum_{A\in\mathcal{J}\colon j\in A}y_{Aj}\geq d_j\quad \forall j\in J\quad\text{ and  }\sum_{j\in A} y_{Aj}\leq q\cdot x_A \quad \forall A\in\mathcal{J}?
\end{align*}
\end{problem}
The interpretation of the problem is as described in the introduction: Can we choose $B$ subsets, 
with multiple choices being allowed since $x_A\in \mathbb{Z}_{\geq 0}$, such that the demand of each element is covered, when
each subset may only cover up to $q$ elements (again multiple choices are allowed since $y_{Aj}\in\mathbb{Z}_{\geq 0}$).
For a fixed subset~$A$, the integer $y_{Aj}$ in the problem definition models the amount of demand of element~$j$ covered by the subset~$A$.

\begin{remark}
\label{rem:multiset_multicover}
If, instead of regarding the subsets~$A\in\mathcal{J}$, we regard all multisets
of cardinality $q$ of~$A$, we get an instance of \emph{Multiset Multicover}, raising 
the input size only by a polynomial factor as $q$ is not part of the input. Thereby,
\nonRobustShort\ is in some sense a representation of certain \emph{Multiset Multicover}
instances, having smaller input size. This connection, however, is lost when regarding the
robust version of \nonRobustShort, cf.\ Section~\ref{sec: Robust Version}.
\end{remark}
In the sequel, it will be useful to consider the following alternative 
definition of \nonRobustShort.

\begin{problem}[\nonRobust~(\nonRobustShort) - alternative definition]
\label{prob: MSMC_alt}
\mbox{}\\
\noindent\textbf{Instance:} 
Finite sets $I, J$ with $I\cap J=\emptyset$, weights $d_j\in\mathbb{Z}_{\geq 0}$ for all $j\in J$, a bipartite graph $G=(I\cup J, E)$ and a positive integer
$B\in\mathbb{Z}_{>0}$.

\noindent\textbf{Question:}
Are there integers $x_i\in\mathbb{Z}_{\geq 0}$ for $i\in I$ with $\sum_{i\in I}x_i\leq B$, 
such that there exist integers $y_{ij}\in\mathbb{Z}_{\geq 0}$ for $i\in I$, $j\in J$ 
satisfying
\begin{align*}
\sum_{i\in N(j)}y_{ij}\geq d_j \quad\forall j\in J\quad\text{ and } \sum_{j\in N(i)} y_{ij}\leq q\cdot x_i \quad\forall i\in I?
\end{align*}
\end{problem}
For an instance of \nonRobustShort,\ we call the set $I$
\emph{locations} and the set $J$ \emph{regions}. Further, $d_j$
describes the number of \emph{clients} or the \emph{demand} 
in region $j\in J$ and $x_i$
denotes the number of \emph{suppliers} in location $i\in I$.  The
number $q$ can be interpreted as the number of clients a single
supplier may serve.  In the optimization version \emph{\nonRobustOpt}
(\nonRobustOptShort) we aim for a minimum number of suppliers.
Identifying each location~$i$ with its neighborhood~$N_G(i)$ yields the equivalence
of the two problem definitions.

It can readily be seen that the following integer program models~\nonRobustOptShort\ for 
some demand vector $d\in\mathbb{Z}_{\geq 0}^{|J|}$:
\begin{mini!}
	{x,y}{\sum_{i\in I}x_i}
	{\label{ip:allocation}}
	{\text{(MIP~\ref{ip:allocation})}(d)\quad}
	\addConstraint{\sum_{i\in N(j)}y_{ij}}{\geq d_j\quad}{\forall j\in J}
	\addConstraint{\sum_{j\in N(i)}y_{ij}}{\leq q\cdot x_i\quad}{\forall i\in I}
	\addConstraint{y_{ij}}{\geq 0}{\forall i\in I,~j\in J}
	\addConstraint{x_i}{\in\mathbb{Z}_{\geq 0}\quad}{\forall i\in I.}
\end{mini!}
Note that the variables~$y_{ij}$ are not forced to be integral. 
Observation~\ref{obs: interpretation y} argues why this is no restriction. Furthermore, in Lemma~\ref{lemma:ipequivalence} we prove that 
(IP~\ref{ip:subset})($d$) is an alternative formulation to (MIP~\ref{ip:allocation})($d$).
\begin{mini!}
	{x}{\sum_{i\in I}x_i}
	{\label{ip:subset}}
	{\text{(IP~\ref{ip:subset})}(d)\quad}
	\addConstraint{\sum_{i\in N(S)}q\cdot x_i}{\geq\sum_{j\in S}d_j\quad}{\forall S\subseteq J}
	\addConstraint{x_i}{\in\mathbb{Z}_{\geq 0}}{\forall i\in I.}
\end{mini!}

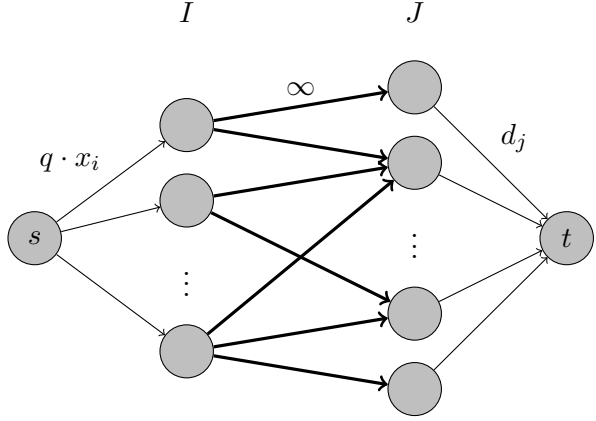
\begin{figure}%
\centering
\begin{tikzpicture}
	\node[vertex] (s) {$s$};
	\node[vertex, right of=s, xshift=1cm, yshift=1.5cm] (i1) {};
	\node[vertex, below of=i1] (i2) {};
	\node[vertex, draw, below of=i2, yshift=-1cm] (i3) {};
	\node[below of=i2] (i4) {$\vdots$};
	\node[vertex, draw, right of=i1, yshift=0.5cm, xshift=2cm] (j1) {};
	\node[vertex, below of=j1] (j2) {};
	\node[below of=j2] (j5) {$\vdots$};
	\node[vertex, below of=j2, yshift=-1cm] (j3) {};
	\node[vertex, below of=j3] (j4) {};
	\node[vertex, right of=s, xshift=6cm] (t) {$t$};
	\node[above of=j1] (J) {$J$};
	\node[left of=J, xshift=-2cm] (I) {$I$};

	\draw[->] (s) -- node[above left] {$q\cdot x_i$} (i1);
	\draw[->] (s) -- (i2);
	\draw[->] (s) -- (i3);
	
	\draw[->,very thick] (i1) -- node[above] {$\infty$} (j1);
	\draw[->,very thick] (i1) -- (j2);
	\draw[->,very thick] (i2) -- (j2);
	\draw[->,very thick] (i2) -- (j3);
	\draw[->,very thick] (i3) -- (j2);
	\draw[->,very thick] (i3) -- (j3);
	\draw[->,very thick] (i3) -- (j4);
	
	\draw[->] (j1) -- node[above right] {$d_j$} (t);
	\draw[->] (j2) -- (t);
	\draw[->] (j3) -- (t);
	\draw[->] (j4) -- (t);
\end{tikzpicture}
\caption{Flow network for the proof of Lemma
  \ref{lemma:ipequivalence}. All thick arcs have infinite capacity.}%
\label{fig:flownetwork}%
\end{figure}

\begin{lemma}
\label{lemma:ipequivalence}
For an instance of \nonRobustOpt\ it holds that $x\in\mathbb{Z}_{\geq{0}}^{|I|}$ is a feasible solution to (IP \ref{ip:subset})($d$) if and only if there exists $y\in\mathbb{R}_{\geq 0}^{|I||J|}$ such that $(x,y)$ is a feasible solution for (MIP~\ref{ip:allocation})($d$). In this case, the variables $y$ can be chosen to be integral.
\end{lemma}

\begin{proof}
If $(x,y)$ is a feasible solution for (MIP~\ref{ip:allocation})($d$) then $x$ is also feasible for (IP~\ref{ip:subset})($d$), as for any $S\subseteq J$ we have: 
\begin{align*}
\sum_{i\in N(S)}q\cdot x_i &\geq \sum_{i\in N(S)} \sum_{j\in N(i)} y_{ij} = \sum_{i\in N(S)} \left(\sum_{j\in N(i)\cap S} y_{ij} + \sum_{j\in N(i)\setminus S} y_{ij}\right)\\
&\geq \sum_{i\in N(S)} \sum_{j\in N(i)\cap S} y_{ij} = \sum_{j\in S} \sum_{i\in N(j)} y_{ij} \geq \sum_{j\in S} d_j.
\end{align*}

Now assume we are given a feasible solution $x$ for (IP~\ref{ip:subset})($d$). 
Let $G$ be the graph from the instance of \nonRobustOptShort. We define a directed graph $H=(I\cup J\cup\{s\}\cup\{t\}, R\cup R_s\cup R_t)$, where $R$ contains all arcs in $E(G)$ directed from $I$ to $J$, $R_s=\{(s,i)\colon i\in I\}$ and $R_t=\{(j,t)\colon j\in J\}$, cf.~Fig.~\ref{fig:flownetwork}. 
We set the capacities of each arc $r\in R(H)$ as
$$ c(r)=\begin{cases} \infty,& r\in R,\\q\cdot x_i,& r\in R_s,\\ d_j, & r\in R_t.\\\end{cases}$$
We claim that the maximum $s$-$t$-flow in $H$ has flow value $\sum_{j\in J}d_j$. 
Note that given an $s$-$t$-flow $f$ with flow value $\sum_{j\in J}d_j$, we can define a feasible solution to (MIP~\ref{ip:allocation})($d$) by $(x,y)$, where $y_{ij}= f(i,j)$ for all $(i,j)\in R$.

Further, the flow value of any $s$-$t$-flow can never be larger than $\sum_{j\in J}d_j$ (consider the $s$-$t$-cut with $T=\{t\}$).
Thus, it suffices to show that a maximum $s$-$t$-flow in $H$ has flow value no less than $\sum_{j\in J}d_j$.
To this end let $S,T\subseteq V(H)$ be any $s$-$t$-cut in $H$. 
Let $J^\prime=J\setminus S$, possibly being the empty set. 
If any location in the neighborhood of $J^\prime$ is contained in $S$, the cut contains an arc with infinite capacity, so assume $N^{-}_H(J^\prime)\cap S = \emptyset$ so that $N^{-}_H(J^\prime) \subseteq T$. 
Since $x$ is a feasible solution to (IP~\ref{ip:subset})($d$) we obtain for any subset $Q\subseteq J$
\begin{align*}
\sum_{i\in N^{-}_H(Q)}q\cdot x_i = \sum_{i\in N_G(Q)}q\cdot x_i \geq \sum_{j\in Q} d_j.
\end{align*}
We get
\begin{align*}
c(S,T)\geq\sum_{j\in J\cap S}d_j + \sum_{i\in N^{-}_H(J^\prime)} q\cdot x_i \geq \sum_{j\in J\cap S}d_j + \sum_{j\in J^\prime} d_j = \sum_{j\in J}d_j.
\end{align*}
Thus, every $s$-$t$-cut has capacity larger or equal to
$\sum_{j\in J}d_j$ and by the \emph{Max-Flow-Min-Cut Theorem} we \textcolor{red}{obtain}
the desired result, cf.~\cite{Ahuja+etal:book}.
\end{proof}

\begin{observation}
\label{obs: interpretation y}
Note that the capacities of the arcs in $R_s$ and $R_t$ defined in the
proof of Lemma~\ref{lemma:ipequivalence} are integral. Thus, there
exists an integral flow $f$ in $H$ if and only if there exists a
continuous flow $f'$ in $H$ and the variable~$y_{ij}$ can be
interpreted as the number of clients in region~$j$ taken over by the
suppliers in location $i$, cf.~\cite{Ahuja+etal:book}.
\end{observation}

In the following, we state our results on the complexity of \nonRobustShort. 
Formal proofs of these claims can be found in~\ref{appendix: complexity}.

\begin{observation}
  \emph{Min-$1$-Multiset Multicover} is solvable in linear time.
\end{observation}

\begin{theorem}
	\label{thm:q2}
  \emph{Min-$2$-Multiset Multicover} can be solved in
  $O(|I|^{5/2}|J|^{5/2})$.
\end{theorem}

\begin{theorem}
  \label{the:msmqc_npcomplete}
  For any fixed $q\geq 3$, \nonRobust\ is NP-complete in the strong sense.
\end{theorem}

\begin{observation}
  There is a $\log(q)$-approximation for \nonRobustOpt.
\end{observation}

Now, we concentrate on a robust version of \nonRobust.


\section{Problem definition and classification of the robust version}
\label{sec: Robust Version}
In this section, we extend the initial problem \nonRobustOpt\ to
include uncertainty in the number of clients~$d_j$ of each region
$j\in J$.  We apply concepts of robust optimization such as strict and adjustable robustness
and combine interval and budgeted uncertainty as mentioned in the introduction, 
cf.~\cite{ROBook,adjustableRobustness}.
For each region $j\in J$, we consider two
non-negative integers $a_j$ and $b_j$ with $a_j \leq b_j$, which respectively correspond to the minimum and maximum number of clients in that region.
Concerning the total amount of clients in all regions, we additionally
require this value not to exceed some given constant
$\Gamma \in \mathbb{Z}_{\geq 0}$ to prevent the global worst case.
Then, a vector $\xi \in \mathbb{Z}^{|J|}$ with $\xi_j\in[a_j,b_j]$ and
$\sum_{j\in J}\xi_j\leq \Gamma$ is called a \emph{scenario} and we
denote by $\mathcal{U}$ the set of all scenarios, i.e.,
\begin{align*}
  \mathcal{U}=\left\{\xi\in\mathbb{Z}^{|J|}:\xi_j\in[a_j,b_j]\ \forall j\in J, \sum_{j\in J}{\xi_j}\leq \Gamma\right\}.
\end{align*}
The set $\mathcal{U}$ is also called the \emph{uncertainty set}.  Note
that $\mathcal{U}$ is finite since we only consider integral demands
in our problem. For a vector $x\in\mathbb{R}^{n}$ and some set $F\subseteq\{1\dots n\}$ we use
the common notation $x(F)=\sum_{i\in F}x_i$.
\begin{assumption}
  \label{assumption: gamma}
  In order to obtain a meaningful uncertainty set we assume that
  $\sum_{j\in J}a_j \leq \Gamma \leq \sum_{j\in J}b_j$ implying
  $\mathcal{U}\neq \emptyset$. Moreover, we assume without loss of
  generality that $\Gamma$ is chosen in such a way that
  $b_j + \sum_{k\neq j}a_k\leq \Gamma$ such that, for each region $j\in J$, there exists a scenario
  $\xi$ with $\xi_j=b_j$. Otherwise we could
  decrease the upper bound~$b_j$ in the corresponding region.
\end{assumption}
The intuition of the robust version of \nonRobustOptShort\ is to choose a minimum number of suppliers, such that in 
any scenario of $\mathcal{U}$ all clients may be served. In the following,
we will see how to incorporate this intuition into the models introduced
in Section~\ref{sec: Problem Definition and Classifications}. We begin by 
\emph{robustifying} (IP~\ref{ip:subset})($d$).

Each scenario $\xi\in\mathcal{U}$ defines a
single problem in the fashion of Problem~\ref{prob: MSMC_alt} when
denoting the amount of clients by $d_j=\xi_j$ for all $j$.  Therefore,
given a fixed scenario $\xi$, we consider the integer linear program (IP~\ref{ip:subset})$(\xi)$.
In terms of robust optimization, we obtain the uncertain integer
linear program:
\begin{equation}
  \label{eq: uncertain IP, set}
  \left\{\min\left\{\,\sum_{i\in I}x_i: x \text{ is feasible for
        (IP~\ref{ip:subset})}(\xi)\,\right\}: \xi \in \mathcal{U}\right\} .
\end{equation}
Our aim is to find $x_i\in\mathbb{Z}_{\geq 0}$ for all $i\in I$,
such that all clients can be
served independently of the actually occurring ``true''
scenario. Therefore, we concentrate on the analysis of the following
problem:
\begin{problem}[\RobustShort, set formulation]
  \label{prob: RMSMqC, set}
  \mbox{}\\
  \noindent\textbf{Instance:} Set of possible locations $I$, set of
  regions $J$, non-negative integers $a_j$, $b_j$ with $a_j\leq b_j$
  for all $j\in J$, integer $\Gamma$ satisfying
  Assumption~\ref{assumption: gamma}, bipartite graph $G=(I\cup J, E)$
  and a positive integer~$B\in\mathbb{Z}_{>0}$.

  \noindent\textbf{Question:} Are there 
  $x_i\in\mathbb{Z}_{\geq 0}$ for all $i\in I$ such that
  $\sum_{i\in I}x_i\leq B$ and for all subsets $S \subseteq J$ and all
  scenarios $\xi \in \mathcal{U}$ we have
\begin{align*}
\sum_{i\in N(S)}q \cdot x_i\geq\sum_{j\in S}\xi_j?
\end{align*}
\end{problem}
The minimization problem corresponding to \RobustShort, i.e. \RobustOptShort, can be formulated as the robust counterpart of \eqref{eq: uncertain IP, set}:
\begin{mini!}|s|[]
{x}
{\sum_{i\in I} {x_i}}
{\label{IP: RMSMqC, set, start}}
{\text{(IP~\ref{IP: RMSMqC, set, start})}\quad}
\addConstraint{\sum_{i \in N(S)} q\cdot x_i}{\geq \sum_{j\in S}\xi_j \quad\label{ip:robust_set_const}}{\forall S\subseteq J,\ \forall\xi \in \mathcal{U}}
\addConstraint{x_i}{\in \mathbb{Z}_{\geq 0} \quad}{\forall i\in I.}
\end{mini!}
Note that the uncertain data only occurs on the right hand side of the
above constraints.  Thus, this problem can be simplified by computing,
for every subset~$S\subseteq J$, the maximum of $\sum_{j\in S} \xi_j$ over
the uncertainty set $\mathcal{U}$.  This maximum is given by
$\tilde{d}_S:=\min\{b(S),\Gamma-a(J\setminus S)\}$: 
\begin{itemize}
\item[-] If $b(S) + a(J\setminus S) \leq \Gamma$, then $\underset{\xi\in\mathcal{U}}{\max}\ \sum_{j\in S}\xi_j = b(S)$.
\item[-] If $b(S) + a(J\setminus S) > \Gamma$, then $\underset{\xi\in\mathcal{U}}{\max}\ \sum_{j\in S}\xi_j = \Gamma - a(J\setminus S)$. (Since
  $b(S) > \Gamma - a(J\setminus S)$ and
  $a(S) \leq \Gamma - a(J\setminus S)$ a corresponding scenario
  clearly exists.)
\end{itemize}
Consequently, we can replace $\sum_{j\in S}\xi_j$ in line~\eqref{ip:robust_set_const}
 of~(IP~\ref{IP: RMSMqC, set, start}) by $\tilde{d}_S$
and reformulate the question posed in Problem~\ref{prob: RMSMqC, set}
as follows: Are there
$x_i\in\mathbb{Z}_{\geq 0}$ for all $i\in I$ such that
$\sum_{i\in I}x_i\leq B$ and for all subsets $S \subseteq J$:
\begin{align*}
\sum_{i\in N(S)}q \cdot x_i\geq \tilde{d}_S?  
\end{align*}
Observe that when using this formulation the problem is independent
from the uncertainty set $\mathcal{U}$. But in comparison to the non-robust
formulation of Section~\ref{sec: Problem Definition and
  Classifications}, the value $\tilde{d}_S$ cannot be split into a sum
of clients over the single regions of $S$ anymore.

As in Section~\ref{sec: Problem Definition and Classifications}, we
aim to obtain an equivalent assignment formulation for
Problem~\ref{prob: RMSMqC, set}.  A first idea is to consider
the robust counterpart of the uncertain IP

\begin{equation}
\label{eq: uncertain IP, assignment}
\left\{\min\left\{\,
    \sum_{i\in I}x_i: (x,y) \text{ is feasible for 
		(MIP~\ref{ip:allocation})}(\xi)\,\right\}: \xi \in \mathcal{U}\right\}. 
\end{equation}

Since $(x,y)$ needs to be feasible for (MIP~\ref{ip:allocation})$(\xi)$
for any scenario $\xi$ and since for each region $j\in J$ there exists
a scenario with $\xi_j=b_j$ (cf.\ Assumption~\ref{assumption: gamma}),
the solution vector $(x,y)$ can be computed by solving the
mixed integer linear program (MIP~\ref{ip:allocation})$(b)$.
In general, $\sum_{j\in S}b_j = b(S) = \tilde{d}_S$
does not hold for all $S\subseteq J$ and we see that formulation 
\eqref{eq: uncertain IP, assignment} cannot be equivalent to \eqref{eq: uncertain IP, set}.
Furthermore, the upper bound on the number of clients $\Gamma$ is not
needed in~\eqref{eq: uncertain IP, assignment}. Eliminating the constraint
$\sum_{j\in J}\xi_j\leq \Gamma$ in the definition of the uncertainty
set $\mathcal{U}$ would lead to $\tilde{d}_S=b(S)$.  Only in this
special case, both formulations \eqref{eq: uncertain IP, assignment} and
\eqref{eq: uncertain IP, set} are equivalent as shown in Section \ref{sec:
  Problem Definition and Classifications}.

Actually, computing a global solution~$y$ is far too conservative and
applying strict robustness is unrewarding.  Moreover, 
\eqref{eq: uncertain IP, assignment} does not match the intuition of \RobustOptShort\ as we have
to fix the $y_{ij}$ before the actual scenario is revealed. Recalling the interpretation of the variables~$y_{ij}$ in Observation~\ref{obs: interpretation y}, it is meaningful to fix the variables $y_{ij}$ only \emph{after} the realization of the true scenario $\xi$ is known. Thus, we only need to settle the decision over the $x_i,$
$i \in I,$ before the realization becomes apparent, while additionally
ensuring the existence of an assignment~$y$ of suppliers to clients. Therefore, we apply the concept of \emph{adjustable robustness}
\cite{adjustableRobustness} with $x$ containing the ``here and now''
variables and $y$ corresponding to the ``wait and see'' variables.
Then, our aim is to find $x_i\in\mathbb{Z}_{\geq 0}$ for all $i\in I$,
minimizing $\sum_{i\in I}x_i$, such that, for every
$\xi \in\mathcal{U}$, there exist $y(\xi)$ with $(x,y(\xi))$ being
feasible for (MIP~\ref{ip:allocation})$(\xi)$.  This approach leads to
the \emph{adjustable robust counterpart} of~\eqref{eq: uncertain IP,
  assignment}:
\begin{mini!}|s|[]
{x,y}
{\sum_{i\in I} {x_i}}
{\label{IP: ARC}}
{\text{(MIP~\ref{IP: ARC})}\quad}
\addConstraint{\sum_{i \in N(j)} y_{ij}(\xi)}{\geq \xi_j \quad}{\forall j\in J,\ \xi \in \mathcal{U} \label{ARC1}}
\addConstraint{\sum_{j \in N(i)} y_{ij}(\xi)}{\leq q \cdot x_i \quad}{\forall i \in I,\ \xi \in \mathcal{U} \label{ARC2}} 
\addConstraint{y_{ij}(\xi)}{\geq 0 \quad}{\forall i\in I,\ \forall j\in J,\ \xi\in \mathcal{U}}
\addConstraint{x_i}{\in \mathbb{Z}_{\geq 0} \quad}{\forall i\in I.}
\end{mini!}
The corresponding decision problem is defined as follows:
\begin{problem}[\RobustShort, assignment formulation]
\label{prob: RMSMqC, assignment}~\\
\noindent\textbf{Instance:} Set of possible locations~$I$, set of regions~$J$, non-negative integers~$a_j$, $b_j$ with
$a_j\leq b_j$ for all $j\in J$, integer~$\Gamma$ satisfying
Assumption~\ref{assumption: gamma}, bipartite graph $G=(I\cup J, E)$
and a positive integer~$B\in\mathbb{Z}_{>0}$.

\noindent\textbf{Question:} Are there
$x_i\in\mathbb{Z}_{\geq 0}$ for all $i\in I$, such that 
$\sum_{i\in I}x_i\leq B$ and for all scenarios $\xi \in \mathcal{U}$ there are
$y_{ij}(\xi)\in\mathbb{R}_{\geq 0}$ for all $i\in I$, $j\in J$, such
that 
\begin{align*}
  \sum_{i\in N(j)}y_{ij}(\xi) \geq \xi_j \ \forall j \in J\text{ and } \sum_{j\in N(i)}y_{ij}(\xi) \leq q \cdot x_i \ \forall i\in I ?   
\end{align*}
\end{problem}
As in the non-robust version, the assignment variables~$y_{ij}(\xi)$
can be chosen to be integral whenever there exists a solution
$(x,y)\in \mathbb{Z}^{|I|}\times\mathbb{R}^{|I||J||\mathcal{U}|}$ for
Problem~\ref{prob: RMSMqC, assignment}.  
Thus, we can interpret the variable
$y_{ij}(\xi)$ as the number of clients in region~$j$ taken over by suppliers in location~$i$ in case scenario $\xi$
occurs. A similar problem is analyzed in ~\cite{GABREL2014100}, whereas a general approach for adjustable robustness in the LP-case with right hand side uncertainty is investigated in~\cite{Minoux2011}.

Now, we are able to prove the equivalence between the robust
set formulation defined in Problem~\ref{prob: RMSMqC, set} and the
adjustable robust assignment formulation defined in Problem~\ref{prob:
  RMSMqC, assignment}.

\begin{proposition}
  Problem~\ref{prob: RMSMqC, set} and Problem~\ref{prob: RMSMqC,
    assignment} are equivalent. 
\end{proposition}
\begin{proof}
  Since the objective functions are the same, it remains to be shown
  that any solution~$(x,y)$ of (MIP~\ref{IP: ARC}) yields a solution
  $x^\prime$ of (IP~\ref{IP: RMSMqC, set, start}) with the same
  objective value and vice versa.  Thus, let $(x,y)$ be feasible for \eqref{IP: ARC} with
  $y=(y(\xi_1), y(\xi_2),\ldots)$.
  Fix a scenario $\xi \in \mathcal{U}$. Then $(x,y(\xi))$ is feasible
  for (MIP~\ref{ip:allocation})$(\xi)$.  Due to the equivalence of the
  formulations in the non-robust version, we get that $x$ fulfills
	\begin{align*}
	\sum_{i \in N(S)} q\cdot x_i\geq \sum_{j\in S} \xi_j\quad \forall S\subseteq J
	\end{align*}
Since this argument holds true for any fixed scenario $\xi$, we obtain that $x$ is feasible for (IP~\ref{IP: RMSMqC, set, start}).

On the other hand, given a solution $x$ of (IP~\ref{IP: RMSMqC, set, start}), for any fixed scenario $\xi$, there exists $y(\xi)$ such that $(x,y(\xi))$ is feasible for (MIP~\ref{ip:allocation})$(\xi)$ due to the results of Section~\ref{sec: Problem Definition and Classifications}.
In total, we obtain that $(x,y)$ with $y=(y(\xi_1), y(\xi_2),\ldots)$ is feasible for (MIP~\ref{IP: ARC}).
\end{proof}

At this point, we see that our initial link to \emph{Multiset Multicover by
$q$-sets} is lost when including robustness, since the assignment
variables $y\in \mathbb{R}^{|I||J||\mathcal{U}|}$ can be specified in
a subsequent step when the ``true'' scenario is already known.  In the
robust case the
value $x_i$ has to be specified in advance so 
that, for each possible scenario $\xi$, there exists a selection of
$x_i$ sets for each location $i\in I$ satisfying the upcoming demand.
Therefore, considering robustness leads to a completely new problem in
comparison to Section~\ref{sec: Problem Definition and
  Classifications} which we investigate further in the following.
Before we concentrate on the complexity of \RobustShort, we state
some properties of the problem.

\begin{observation}
  Let $z$ be the optimal solution value of \RobustOptShort. Then:
  \begin{itemize}
  \item[(a)] $z \geq \left\lceil \frac{\Gamma}{q} \right\rceil.$ 
  \item[(b)] Define $\bar{x}\in\mathbb{Z}^{|I|}$ in the following way:
	For all $j\in J$, choose $i\in N(j)$ and increase $\bar{x}_i$ by $\left\lceil \frac{b_j}{q} \right\rceil$. Then, $\bar{x}$ is feasible for (IP~\ref{IP: RMSMqC, set, start}) and we get 
	$z \leq \sum_{j\in J}\left\lceil \frac{b_j}{q} \right\rceil.$
  \item[(c)] It suffices to consider all scenarios $\xi\in\mathcal{U}$
    with $\sum_{j\in J}\xi_j=\Gamma.$  
  \end{itemize}
\end{observation}
From part~(a) and part~(b) it follows that \RobustOptShort\ has a finite optimal solution. Furthermore, from now on we restrict the problem to scenarios whose demands sum up to $\Gamma$. We call such a scenario an \emph{extreme scenario} and
denote by $\mathcal{U}^\prime\subseteq \mathcal{U}$ the set of all
extreme scenarios.

In the following, we utilize the \emph{Dominating Set} problem~\cite{Garey}
to show NP-hardness for \RobustShort. In the former
problem, given an undirected graph $G=(V,E)$ and a positive integer
$K\leq |V|$, the question is whether there exists a subset
$V^\prime\subseteq V$ with $|V^\prime|\leq K$ such that for all
$u\in V\setminus V^\prime$ there is $v\in V^\prime$ for which
$\{u,v\}\in E$. This problem is well-known to be NP-complete.

%

\begin{theorem}
  For fixed $q\in\Z_{>0}$, \RobustShort\ is strongly NP-hard.
\end{theorem}
\begin{proof}
  We show that there exists a polynomial time reduction from
  \emph{Dominating Set} to \RobustShort.  To this end, let an undirected
  graph $G=(V,E)$ with $V=\left\{1,\ldots,n\right\}$ and an integer $K\leq n$ be
  given.  To construct an instance of \RobustShort\ we set
  $I=\left\{1,\ldots,n\right\}$ and $J=\left\{n+1,\ldots,2n\right\}$.  For every edge
  $\{u,v\}\in E$, we add the edge $\left\{u,n+v\right\}$ and the edge $\{v,n+u\}$
  to the bipartite graph $G^\prime$ with vertex set $I\cup J$ and edge
  set $E'$.  Additionally, for every $v\in V$, the edge $\{v,n+v\}$ is
  added to $E'$.  Moreover, we define $a_j=0$, $b_j=1$ for all
  $j\in J$, $\Gamma =1$ and $B=K$.  Thus, we have $\tilde{d}_S=\min\left\{|S|,1\right\}=1$
  for any non-empty subset $S\subseteq J$.

  Let $V^\prime\subseteq V=I$ be a solution of \emph{Dominating Set} such
  that $|V^\prime|\leq K$. Then, we set $x_i=1$ for all
  $i\in V^\prime$ and zero else so that $\sum_{i\in I}x_i\leq B$
  already holds true. Fix an arbitrary subset $S\subseteq J,$
  $S \neq \emptyset$.  We want to show that
  \begin{align*}
    \sum_{i\in N(S)}q \cdot x_i \geq 1.    
  \end{align*}
  Thus, we need to prove that at least one value $x_i$ for $i\in N(S)$ is set to one, i.e. $N(S)\cap V'\neq \emptyset$.
  Choose an arbitrary element $n+v \in S$ with $v\in
  \left\{1,\ldots,n\right\}$. $V^\prime$ is a dominating set, so we have
  $v\in V^\prime$ or there is $u\in V^\prime$ adjacent to $v$ in $G$.
  In the former case, $v \in N(S)$ since $G^\prime$ contains the edge
  $\{v,n+v\}$.
  In the latter case, $u\in N(S)$ since $G^\prime$ contains the edge $\{u,n+v\}.$ Thus, $\sum_{i\in N(S)}q\cdot x_i \geq 1$ holds true in any case, so that $x$ is a solution of \RobustShort.
 
  Conversely, suppose that $x$ is a solution of \RobustShort\ such
  that $\sum_{i\in I}x_i\leq B$.  Since $\tilde{d}_S=1$ for all
  $S\subseteq J,\ S\neq \emptyset$, we can assume without loss of
  generality that $x_i\leq 1$ for all $i\in I.$ The set $V^\prime$ is
  defined to contain all vertices $v\in V$ such that $x_v=1$.
  Clearly, $|V^\prime|\leq B=K$ and we claim that $V^\prime$ is a
  dominating set for $G$.  To this end, choose a vertex
  $u \in V-V^\prime$ and consider the set $S=\left\{n+u\right\}\subseteq J$.
  Since $x$ is a feasible solution, there is $v\in N(S)$ with $x_v=1$,
  i.e. $v\in V^\prime.$ Since $v\in N(S),$ either $u=v$ or the
  vertices~$u$ and $v$ are adjacent in $G$ by construction of
  $G^\prime$ yielding the claim.
\end{proof} 
Note that we did not prove NP-completeness of \RobustShort. In the
following section, we see that checking a given vector~$x$ for
feasibility is co-NP-complete.


\section{Solving the \RobustOpt}
\label{sec: Solving the Robust Version}
In the previous section we have shown that \RobustOptShort\ is an
NP-hard problem.  As both formulations of the problem as (mixed) integer
linear programs contain a large number of constraints, it is reasonable
to apply constraint generation to obtain a solution,
cf.~\cite{Desaulniers+etal:book,Groetschel:book,Nemhauser+Wolsey:IP:book,Wolsey+Pochet:book}. 

Thus, focusing on the set formulation, at any point in the constraint
generation process, a collection of subsets $\mathcal{S} \subseteq 2^J$ is
given and we solve the relaxed problem obtained by only considering
the constraints corresponding to sets $S\in \mathcal{S}$ in
(IP~\ref{IP: RMSMqC, set, start}).  This problem is called \emph{restricted
  master problem}.  In the \emph{separation step}, given an optimal solution
$\bar{x}$ of the restricted master problem, we are looking for a new
subset $S\subseteq J$ such that the constraint induced by $S$ is not
fulfilled yet, i.e.,
\begin{align*}
  \sum_{i\in N(S)}q\cdot \bar{x}_i < \tilde{d}_S.  
\end{align*}
In the next iteration, $\mathcal{S}$ is updated by adding the newly
found set $S$ and the restricted master problem is solved once more.
If there exists no set $S$ fulfilling the above inequality, we know
that $\bar{x}$ is the optimal solution for \RobustOptShort.
Initially, $\mathcal{S}$ is the empty set yielding the optimal
solution $\bar{x}_i=0$ for all locations $i\in I$ in the restricted master problem.
Analogously, this procedure can be applied to the assignment
formulation (MIP~\ref{IP: ARC}) using an (initially empty) set
$\mathcal{U}''\subseteq\mathcal{U}'$ of extreme scenarios.  The
important step of these methods is an efficient way to solve the
occurring separation problems. These can be formulated as
follows:

\begin{problem}[Separation for \RobustShort, set formulation]
\label{problem: pricing, set}
\mbox{}\\
\noindent\textbf{Instance:} Set of possible locations $I$,
non-negative integers $\bx_i$ for all $i\in I$, set of regions $J$,
non-negative integers $a_j$, $b_j$ with $a_j\leq b_j$ for all $j\in
J$, integer $\Gamma$ satisfying Assumption~\ref{assumption: gamma},
bipartite graph $G=(I\cup J, E)$. 

\noindent\textbf{Question:} Is there a subset $S\subseteq J$ such that 
 $q \cdot \bx(N(S)) < \tilde{d}_S?$
\end{problem}

\begin{problem}[Separation for \RobustShort, assignment formulation]
\label{problem: pricing, assignment}
\mbox{}\\
\noindent\textbf{Instance:} See Problem~\ref{problem: pricing, set}. 

\noindent\textbf{Question:} Is there an extreme scenario $\xi \in
\mathcal{U}^\prime$ such that there is no $y\geq 0$ with  
\begin{align}
\label{eq: pricing, assignment}
\sum_{i\in N(j)}y_{ij}\geq \xi_j \quad\forall j\in J \text{ and }\sum_{j \in N(i)}y_{ij} \leq q \cdot \bx_i \quad\forall i\in I?
\end{align}
\end{problem}

In the following, we concentrate on the analysis of these two problems.
Using Farkas' Lemma (see e.g., \cite{Groetschel:book,Schrijver:book}),
Problem~\ref{problem: pricing, assignment} asks for an extreme
scenario $\xi \in \mathcal{U}^\prime$ such that there are vectors $\mu, \nu \geq 0$ with $\mu_i \geq \nu_j$ for all regions $j\in J$ and locations $i\in N(j)$ and   
\begin{align*}
  \sum_{i\in I} q\cdot \bx_i \cdot \mu_i < \sum_{j\in J}\xi_j\cdot \nu_j.
\end{align*}
Therefore, we can also solve
Problem~\ref{problem: pricing, assignment} by asking for an extreme
scenario $\xi$ such that the optimal objective value of the problem
\begin{mini!}|s|[]
{\mu,\nu}
{\sum_{i\in I} q\cdot \bx_i \cdot \mu_i - \sum_{j\in J}\xi_j\cdot \nu_j}
{\label{LP: pricing}}
{\text{(LP~\ref{LP: pricing})}\quad}
\addConstraint{\mu_i}{\geq \nu_j \quad}{\forall j\in J, i\in N(j)}
\addConstraint{\mu_i, \nu_j}{\geq 0 \quad}{\forall i\in I, j\in J,}
\end{mini!}
is less than zero. Note that the zero vector is feasible here, so that the optimal objective value never exceeds zero.

\begin{definition}
  Let an instance of the separation problem be given.  A set
  $S\subseteq J$ with $q\cdot \bx(N(S)) < \tilde{d}_S$ is called
  \emph{violating} subset.  Analogously, an extreme scenario
  $\xi\in\mathcal{U}'$ such that (LP~\ref{LP: pricing}) has a solution
  $(\mu,\nu)$ with objective value less than zero is called
  \emph{violating} scenario.
\end{definition}

The following Lemma \ref{lemma: pricing} is an easy consequence 
from the equivalence of the Problems \ref{prob: RMSMqC, set} and 
\ref{prob: RMSMqC, assignment}. 
We will nevertheless give a constructive proof. This will enable us to 
find a violating scenario in polynomial time if we are given a violating 
subset and vice versa.

\begin{lemma}
  \label{lemma: pricing}
  Let an instance of the separation problem be given.  Then, there
  exists a violating scenario $\xi \in \mathcal{U}^\prime$ if and only
  if there exists a violating subset $S\subseteq J$. 
\end{lemma}

\begin{proof}
  Let $\xi \in \mathcal{U}^\prime$ and $(\mu,\nu)$ such that
  $(\mu,\nu)$ is feasible for (LP~\ref{LP: pricing}) with objective
  value
  $\sum_{i\in I} q\cdot \bx_i \cdot \mu_i - \sum_{j\in J}\xi_j\cdot
  \nu_j < 0$ be given.  Since the corresponding constraint matrix is
  totally unimodular, we can assume $\mu$ and $\nu$ to only contain
  integral values.  First of all, suppose
  $\nu^\star \coloneqq \max\{\nu_j\colon j\in J\} > 1$ and consider the index set
  $J^\star\coloneqq \{j\in J\colon \nu_j = \nu^\star\}$.  Then,
  $\mu_i \geq \nu^\star$ for all $i\in N(J^\star)$ and without loss of
  generality we can assume that even equality holds.  Now, we obtain
  \begin{align*}
    \sum_{i\in N(J^\star)} q \cdot \bx_i \cdot \mu_i - \sum_{j\in
    J^\star}\xi_j \cdot \nu_j = \left(\sum_{i\in N(J^\star)} q \cdot
    \bx_i - \sum_{j\in J^\star}\xi_j \right) \cdot \nu^\star \eqqcolon A \cdot
    \nu^\star. 
  \end{align*}
  If $A<0$ we can choose $S=J^\star$ and have found a solution for the
  separation problem of the set formulation since
  \begin{align*}
    0 > A &=\sum_{i\in N(J^\star)} q \cdot \bx_i - \sum_{j\in J^\star}\xi_j
          \geq \sum_{i\in N(J^\star)} q \cdot \bx_i - \min\left\{b(J^\star),\Gamma -a(J\setminus J^\star)\right\}\\
          &= \sum_{i\in N(J^\star)} q \cdot \bx_i - \tilde{d}_{J^\star}.
  \end{align*}
  Otherwise, if $A\geq 0$, we can decrease $\mu_i$ for all
  $i\in N(J^\star)$ and $\nu_j$ for all $j\in J^\star$ by one and have
  found another feasible solution $(\mu',\nu')$ with objective value
  smaller than $0$ since
  \begin{align*}
    0 &> \sum_{i\in I}q\cdot \bx_i\cdot\mu_i - \sum_{j\in J}\xi_j \cdot \nu_j\\ 
      &= \left( \sum_{i\in N(J^\star)}q\cdot \bx_i - \sum_{j\in J^\star}\xi_j\right) \cdot \nu^\star + \sum_{i\notin N(J^\star)}q \cdot \bx_i\cdot\mu_i - \sum_{j\notin J^\star}\xi_j\cdot \nu_j\\
      &\geq \left( \sum_{i\in N(J^\star)}q\cdot \bx_i - \sum_{j\in J^\star}\xi_j\right) \cdot \left(\nu^\star-1\right) + \sum_{i\notin N(J^\star)}q \cdot \bx_i\cdot\mu_i - \sum_{j\notin J^\star}\xi_j \cdot \nu_j\\
      &= \sum_{i\in I}q\cdot \bx_i\cdot\mu'_i - \sum_{j\in J}\xi_j\cdot\nu'_j.
  \end{align*}
  Thus, repeating this argument, we either find the desired violating
  set $S$ or we end with a binary solution $\mu_i,\nu_j\in \{0,1\}$
  for all $i,j$.  In the latter case set $S=\{j\in J:\nu_j=1\}$.
  Then, it holds true that $\mu_i\geq 1$ for all $i\in N(S)$.  With no
  loss of generality, we can assume $\mu_i=1$ for $i\in N(S)$ and
  $\mu_i=0$ else since the objective value only becomes smaller.  This
  yields
  \begin{align*}
    0 > \sum_{i \in I} q \cdot \bx_i \cdot \mu_i - \sum_{j\in J}\xi_j \cdot \nu_j
    = \sum_{i\in N(S)} q \cdot \bx_i - \sum_{j \in S}\xi_j
    \geq \sum_{i\in N(S)} q \cdot \bx_i - \tilde{d}_S
  \end{align*}
  and we have found the desired set~$S$.

  On the other hand, given $S\subseteq J$ with
  $\sum_{i\in N(S)} q \cdot \bx_i < \tilde{d}_S$, we choose an extreme
  scenario $\xi \in \mathcal{U}^\prime$ with
  $\sum_{j\in S}\xi_j=\tilde{d}_S$:
  \begin{itemize}
  \item[-] If $b(S) + a(J\setminus S) \leq \Gamma$, we set
    $\xi_j=b_j$ for all $j\in S$. Since $a(J\setminus
    S)\leq\Gamma - b(S)$ and $b(J\setminus S)\geq \Gamma - b(S)$
    we can choose the demands in the remaining regions $j\in
    J\setminus S$ so that $\sum_{j\in J}\xi_j=\Gamma.$ 
  \item[-] If $b(S) + a(J\setminus S) > \Gamma$, we set
    $\xi_j=a_j$ for all $j\in J\setminus S$. Since $a(S)\leq
    \Gamma - a(J\setminus S)$ and $b(S)>\Gamma-a(J\setminus S)$
    we can choose the demands in the remaining regions $j\in S$
    so that $\sum_{j\in J}\xi_j=\Gamma.$ 
  \end{itemize}
  Finally, set $\nu_j=1$ for all $j\in S$ as well as $\mu_i=1$ for all $i\in N(S)$.
  All other variables are set to zero.
  This yields the desired extreme scenario $\xi$ and the solution $(\mu,\nu)$ for the separation problem in the assignment formulation with objective value
  \begin{align*}
    \sum_{i\in I}q\cdot \bx_i\cdot\mu_i - \sum_{j\in J}\xi_j\cdot\nu_j =
    \sum_{i\in N(S)} q \cdot \bx_i - \tilde{d}_S < 0. 
  \end{align*}
  This completes the proof.
\end{proof}

Lemma~\ref{lemma: pricing} allows to switch between both separation
problems as the proof reveals how to construct a violating scenario
from a given violating subset and vice versa.  Furthermore, concerning
complexity the problems are equally hard since the transformations can
be computed in polynomial time.\ Recall that the polyhedron of
feasible solutions corresponding to (LP~\ref{LP: pricing}) is integral,
even if the constraints $\mu_i,\nu_j\leq 1\ \forall i\in I, j\in J$ are added,
since this does not destroy total unimodularity.

In the following, we show that Problem~\ref{problem: pricing, set} is
NP-complete.  To prove this fact, we additionally need the definition
of the \emph{Knapsack} problem, cf.~\cite{Garey}. Here, we are given a finite set~
$U$, for each element $u \in U$ a size $s(u)\in \mathbb{Z}_{>0}$ and a
profit $p(u)\in \mathbb{Z}_{>0}$, and two positive integers $B$ and
$K$. The question is whether there exists a subset
$U^\prime\subseteq U$ such that $\sum_{u\in U '}s(u)\leq B$ as well as
$\sum_{u\in U'}p(u)\geq K$. In the following we write $s(U')$,
respectively $p(U')$, to refer to the sum over the sizes/profits of
the single elements in~$U'$.

%

\begin{lemma}
  \label{lemma: pricing NP}
  For fixed $q\in\Z_{>0}$, Problem~\ref{problem: pricing, set} is NP-complete.
\end{lemma}

\begin{proof}
  Clearly, Problem~\ref{problem: pricing, set} is contained in NP
  since given any subset $S\subseteq J$ we can check in polynomial
  time whether the stated inequality is satisfied.

  We show that \emph{Knapsack} reduces to Problem~\ref{problem: pricing, set}
  in polynomial time.  Let an arbitrary instance of \emph{Knapsack} be given
  with a set $U=\{1,\ldots,n\}$, sizes $s(u)$ and profits $p(u)$
  associated with each element $u\in U$ and two integers $B,K\in\Z_{>0}$. We define a bipartite graph
  $G=(I \cup J,E)$ with $I=\left\{1,\ldots,n,2n+1\right\}= U \cup \left\{2n+1\right\}$, $J=\left\{n+1,\ldots,2n,2n+2\right\}$ and
  \begin{align*}
  E=&\left\{\{u,n+u\},\{2n+1,n+u\}, \{u,2n+2\} \text{ for }u=1,\ldots,n\right\}\\
  &\cup\left\{\{2n+1,2n+2\}\right\}.
  \end{align*}
  Furthermore, we set $\bx(u)=s(u)$ and $b(n+u)=q\cdot (p(u)+s(u))$ for all $u\in U$. 
	Further, we set $\bx(2n+1)=K-1$, $b(2n+2)=0$, $\Gamma = q\cdot(B+K)$.
	Finally, we set $a(j)=0$ for all $j\in J$.

  Now given a solution $U^\prime \subseteq U \subseteq I$ of \emph{Knapsack}
  with $s(U^\prime)\leq B$ and $p(U^\prime)\geq K$, we choose $S =
  \{n+u\colon u \in U^\prime\}\subseteq J$. 
  Then, $S$ is nonempty since $U^\prime \neq \emptyset$ and we have
  $\nicefrac{\Gamma}{q} - \bx(U^\prime) = B + K - s(U^\prime) \geq B + K - B = K$
  and  
  \begin{align*}
    \frac{b(S)}{q} - \bx(U^\prime) &= \sum_{u\in U^\prime}\frac{b(n+u)}{q} - \sum_{u \in U^\prime}\bx(u) = \sum_{u\in U^\prime}p(u) + s(u) - \sum_{u \in U^\prime}s(u)\\
                       &=  \sum_{u\in U^\prime}p(u)  \geq K,
  \end{align*}
  yielding $\min\left\{\nicefrac{b(S)}{q}, \nicefrac{\Gamma}{q}\right\} - \bx(U^\prime) \geq K$.
  Subtracting $K-1=\bx(2n+1)$ on both sides we obtain:
  \begin{align*}
   & \min\left\{\nicefrac{b(S)}{q}, \nicefrac{\Gamma}{q}\right\} - \bx(U^\prime) - \bx(2n+1) \geq 1 > 0\\
    \Leftrightarrow \quad & \min\left\{\nicefrac{b(S)}{q}, \nicefrac{\Gamma}{q}\right\} - \bx(N(S)) \geq 1 > 0,
  \end{align*}
  i.e. $\min\left\{b(S),\Gamma\right\} - q\cdot \bx(N(S))>0$.
  Thus, $S$ is a solution for Problem~\ref{problem: pricing, set}.

  On the other hand, let $S\subseteq J$ be a solution for 
  Problem~\ref{problem: pricing, set} with the property 
  $\min\{b(S),\Gamma\} > q\cdot \bx(N(S))$.  Then, $S$ must contain an element
  of the form $n+u$ for some $u\in U$, since if $S=\{2n+2\}$ the
  inequality is not fulfilled.  Set $S^\prime = S\setminus \left\{2n+2\right\}$
  and $U^\prime = \{u: n+u \in S^\prime\}\subseteq U$.  Our aim is to
  show that $U^\prime$ is a solution for \emph{Knapsack}.  We have
  $S^\prime\neq\emptyset$ and $N(S^\prime)= U^\prime \cup \{2n+1\}$.
  Moreover, it also holds true that $\min\{b(S^\prime),\Gamma\} > q\cdot \bx(N(S^\prime))$,
  since $b(2n+2)=0$, $\bx(N(S^\prime)) \leq \bx(N(S))$ and $q>0$.
  Reformulating the right hand side we get $\bx(N(S^\prime)) = \bx(U^\prime) + \bx(2n+1) = \bx(U^\prime) + K - 1$,
  so that in total we have $\min\{b(S^\prime),\Gamma\} - q\cdot \bx(U^\prime) > q\cdot(K - 1)$,
  i.e., $\min\{\nicefrac{b(S^\prime)}{q},\nicefrac{\Gamma}{q}\} - \bx(U^\prime) \geq K$.  When
  inserting the above definitions this expression becomes
  \begin{equation}
    \min\left\{p(U^\prime)+s(U^\prime),B + K \right\} - s(U^\prime) \geq K.
    \label{eq: blubb}
  \end{equation}
  Now, we need to differentiate between two cases: If $p(U^\prime)+s(U^\prime) \leq B+K$, \eqref{eq: blubb} yields $p(U^\prime) = p(U^\prime) + s(U^\prime) - s(U^\prime) \geq K$
     and $s(U^\prime) \leq B+K - p(U^\prime) \leq B+K-K
      =B.$ If $p(U^\prime)+s(U^\prime)> B+K$, \eqref{eq: blubb} yields $B + K - s(U^\prime) \geq K$,
          i.e., $s(U^\prime) \leq B$. Furthermore, $p(U^\prime) > B + K- s(U^\prime) \geq B + K - B = K$.
  Thus, $U^\prime$ is a solution for \emph{Knapsack}.
\end{proof}

\begin{corollary}
	Problem~\ref{problem: pricing, assignment} is NP-complete.
\end{corollary}
\begin{proof}
	Clearly the preceding Lemma~\ref{lemma: pricing NP} and
  Lemma~\ref{lemma: pricing} imply NP-hardness for
  Problem~\ref{problem: pricing, assignment}.
	Furthermore, Problem~\ref{problem: pricing, 
  assignment} is contained in NP since we only need to solve a
  linear program when given a scenario $\xi$ as a certificate,
  yielding NP-completeness in total.
\end{proof}

Moreover, Lemma~\ref{lemma: pricing NP} reveals that checking whether
a given vector~$\bx$ is feasible for a given instance of \RobustShort\
is co-NP-complete: $\bx$ is feasible if and only if the answer to the
separation problem is ``no''.  

Therefore, we start with a BIP formulation for the set formulation to
solve the separation problem.  Let $z_j$ be one if region $j$ is
contained in $S$ and zero otherwise.  Furthermore, let $y_i$ be one if
$i \in N(S)$ and zero else.  It is easy to see that
Problem~\ref{problem: pricing, set} can be formulated as the following
binary program:
\begin{mini!}|s|[]
{y,z,d}
{\sum_{i\in I}q\cdot \bx_i\cdot y_i - d}{}
{\label{pricing set formulation}}
\addConstraint{d}{\leq \sum_{j\in J}{b_j\cdot z_j}}{}
\addConstraint{d}{\leq \Gamma - \sum_{j\in J}a_j + \sum_{j\in J}{a_j\cdot z_j}}{}
\addConstraint{y_i}{\geq z_j \quad}{\forall j\in J,\ i \in N(j)}
\addConstraint{y_i, z_j}{\in \{0,1\} \quad}{\forall i\in I,\ j\in J},
\end{mini!}
where $\bx$ is the given fixed vector which we wish to test for
feasibility.  This program can be simplified by setting
$\Gamma^\prime\coloneqq\Gamma - \sum_{j\in J}a_j$ and letting $k$ be the
maximum number of regions a location can serve, i.e., the maximum
degree among the vertices~$I$ in the bipartite graph $G=(I\cup J,E)$.
Then, the above formulation can be rewritten:
\begin{mini!}|s|[]
{y,z,d}
{\sum_{i\in I}q\cdot \bx_i\cdot y_i - d}
{\label{pricing set formulation_2}}
{\text{(MIP~\ref{pricing set formulation_2})}\quad}
\addConstraint{d}{\leq \sum_{j\in J}{b_j\cdot z_j}}{}
\addConstraint{d}{\leq \Gamma^\prime + \sum_{j\in J}{a_j\cdot z_j}}{}
\addConstraint{k\cdot y_i}{\geq \sum_{j\in N(i)}z_j \quad}{\forall i\in I}
\addConstraint{y_i, z_j}{\in \{0,1\} \quad}{\forall i\in I,\ j\in J.}
\end{mini!}
Concerning the assignment formulation, we consider (MIP~\ref{pricing
  LP quadratic}) obtained from (LP~\ref{LP: pricing}) by including the
constraints on~$\xi$.  Note that the objective becomes now a quadratic
and non-convex function since $\xi_j$ has changed into a variable.
\begin{mini!}|s|[]
{\mu,\nu,\xi}
{\sum_{i\in I} q\cdot \bx_i \cdot \mu_i - \sum_{j\in J}\xi_j\cdot \nu_j}
{\label{pricing LP quadratic}}
{\text{(MIP~\ref{pricing LP quadratic})}\quad}
\addConstraint{\mu_i}{\geq \nu_j \quad}{\forall i\in I, j\in N(i) \label{first}}
\addConstraint{a_j \leq \xi_j}{\leq b_j \quad}{\forall j\in J}
\addConstraint{\sum_{j\in J}\xi_j}{= \Gamma}{}
\addConstraint{\mu_i, \nu_j}{\geq 0 \quad}{\forall i\in I, j\in J}
\addConstraint{\xi_j}{\in \mathbb{Z}\quad}{\forall j\in J.}
\end{mini!}
When forcing $\nu_j\in\{0,1\}$ (which we can do without loss of
generality, see the proof of Lemma~\ref{lemma: pricing}), we can use the Big-M
method to regain a linear objective.  Once $\nu_j\in\{0,1\}$, the
variables $\mu_i$ can be assumed to be in~$\{0,1\}$ without loss of
generality.  Then, the interpretation of the variables $\mu_i$ and
$\nu_j$ equals that of $y_i$ and $z_j$ and
constraint~\eqref{first} can be simplified in the same manner as
before. In total we obtain:
\begin{mini!}|s|[]
{\mu,\nu,\omega,\xi}
{\sum_{i\in I} q\cdot \bx_i \cdot \mu_i - \sum_{j\in J}\omega_j}
{\label{pricing LP big M}}
{\text{(IP~\ref{pricing LP big M})}\quad}
\addConstraint{k\cdot \mu_i}{\geq \sum_{j\in N(i)}\nu_j \quad}{\forall i\in I}
\addConstraint{\omega_j}{\leq \xi_j\quad}{\forall j\in J}
\addConstraint{\omega_j}{\leq \Gamma\cdot \nu_j\quad}{\forall j\in J}
\addConstraint{a_j \leq \xi_j}{\leq b_j \quad}{\forall j\in J}
\addConstraint{\sum_{j\in J}\xi_j}{= \Gamma}{}
\addConstraint{\mu_i, \nu_j}{\in \{0,1\} \quad}{\forall i\in I, j\in J}
\addConstraint{\xi_j}{\in \mathbb{Z}\quad}{\forall j\in J.}
\end{mini!}
By Lemma~\ref{lemma: pricing} it suffices to solve (IP~\ref{pricing LP
  big M}) since the optimal objective function value of
(MIP~\ref{pricing LP quadratic}) is less than zero if and only if that
of (IP~\ref{pricing LP big M}) is less than zero.  In the optimal
solution of (IP~\ref{pricing LP big M}), the variables $\xi_j$ will be
chosen in order to maximize the term $\sum_{j\in J}\omega_j$ which is
the same as $\tilde{d}_S$ using the new interpretation of $\nu_j$.
Thus, the optimal objective function values of (MIP~\ref{pricing set
  formulation_2}) and (IP~\ref{pricing LP big M}) coincide and we can
also use the proof of Lemma~\ref{lemma: pricing} to solve
(IP~\ref{pricing LP big M}): In the first step we solve
(MIP~\ref{pricing set formulation_2}) yielding the optimal solution
$(y,z,d)$ with objective function value $\Theta$.  If $\Theta=0$, we
obtain the optimal solution of (IP~\ref{pricing LP big M}) by setting
$\mu=\nu=0$ and choosing an arbitrary extreme scenario
$\xi\in\mathcal{U}'$ (which is also the optimal solution of
(MIP~\ref{pricing LP quadratic})).  If $\Theta<0$, we set $\mu=y$ and
$\nu=z$ and choose the extreme scenario $\xi$ depending on whether
$b(S)+a(J\setminus S)$ exceeds $\Gamma$ or not where
$S=\{j\in J\colon z_j=1\}$ (cf.\ proof of Lemma~\ref{lemma: pricing}). On the other hand, 
it is easy to see that a solution~$(\mu,\nu,\xi)$ of~(IP~\ref{pricing LP big M}) 
translates to a solution~$(y,z,d)$ of~(MIP~\ref{pricing set formulation_2}).

\section{Computational results}
\label{sec: Computational Results}
After having analyzed \RobustOpt\ theoretically, in this section we 
present some computational results with $q$ being fixed to three exemplarily.  
The results consist of two main parts: In the first part randomly 
created instances are considered and analyzed, whereas in the second part 
we model a real world problem as \RobustOptShortT\ and display computational 
results based on real world data.  For all instances of \RobustOptShortT, 
we apply both solution approaches based on constraint generation as 
introduced in Section \ref{sec: Solving the Robust Version}.
The solution approach based on the \emph{assignment formulation}, cf.~(MIP~\ref{IP: ARC}), is referred 
to as \emph{asf} while the approach using the \emph{set formulation}, cf.~(IP~\ref{IP: RMSMqC, set, start}), is referred 
to as \emph{setf}. 
For the separation of \emph{asf} we used (IP \ref{pricing LP big M}) and for the 
separation of \emph{setf} we used (MIP~\ref{pricing set formulation_2}). 
We also tested combining the separation approaches using Lemma~\ref{lemma: pricing}, 
but the difference of the running times was neglectably small.

To investigate the price of robustness, we compare the objective value of a solution 
of \RobustOptShortT\ to the objective value of an \emph{average solution}. We obtain the objective value of the 
average solution by randomly choosing a fixed number of possible
scenarios with total demand~$\Gamma$, solving the corresponding non-robust version and determining the 
median of all these objective values. 
Furthermore, we compare the robust solution to the 
solution in which the worst case 
in all regions is assumed, i.e., the number of clients in each region is set 
to the upper bound. We refer to this solution by \emph{worst case solution}.
The objective value of the robust version will be
significantly smaller for most instances than the one of the worst case 
solution.
Although this comparison may seem irrelevant as the 
worst case scenario does not even exist in the robust concept,
the comparison does indeed give insight: Fix for the moment
one location $i\in I$. If the total demand of its neighborhood $N(i)$ is at most
$\Gamma$, there exists a scenario (that needs to be covered) in which 
the demand of each region in $N(i)$ is at its upper bound. If this situation
holds for all locations $i$, it is possible that the worst case is in some
sense simulated by the constraints arising from such scenarios. As an extreme example
regard the case where the neighborhoods of all locations are disjoint. In this case it holds
true that the robust solution value is the same as the worst case solution value, whereas
the average case solution value may be significantly smaller.
Additionally, we analyze the running times of both approaches \emph{asf} 
and \emph{setf}.

To solve the integer linear and mixed integer linear programs, 
the Gurobi Optimizer~8.01 \cite{gurobi} with the Python Interface 
(Python Software Foundation, \texttt{https://www.python.org}) was used.  
All computations were done on a machine with an Intel(R) Xeon(R) CPU E5-2690 0 @ 2.90GHz,
16~cores and 192 GB main memory. The operation system is Ubuntu 64-Bit. 
We permitted each solution approach to use four threads. Further, for the real-world
instances, we aborted (if not finished) the computation after $15$~min wall-clock time and
for the random instances after $2$~min. Computation times are all processing times given in seconds. In the following
two sections, we explain the creation of the instances for the random case and the real world case. Furthermore,
we present and interpret some interesting results.

\subsection{Random instances}
Our random instances are created as follows: We fix the number of regions to $100$ and choose the number of locations from the set $\{10, 20, 30\}$.  The bipartite graph of the instance with edge probability~$p$ is then generated in two steps. To avoid infeasible instances we first randomly choose a location for each region, so that the demand of each region can be covered. In a second step, we add each possible remaining edge independently with  probability $\frac{|I|p - 1}{|I|-1}$ such that the expected number of edges in the created graph is exactly $|I||J|p$. The lower bound $a_j$ for the number of clients in each region is picked uniformly at random from the fixed discrete interval $[0,k_1]$. To obtain the corresponding upper bound $b_j$ a random integer taken from the interval $[1,k_2]$ is added to $a_j$ in a second step. Furthermore, we define the bound~$\Gamma$ to be
\begin{align*}
\Gamma =  \sum_{j\in J} a_j + \left\lfloor d \cdot (\sum_{j\in J}b_j - \sum_{j\in J}a_j) \right\rceil
\end{align*}
for $d\in \left\{\nicefrac{i}{10}, 0 \leq i\leq 10\right\}$. Thus, for $d=0$, the robust solution corresponds to the best case solution since $\mathcal{U}=\left\{(a_1,\dots, a_{|J|})\right\}$ and, for $d=1$, it corresponds to the worst case solution ($\mathcal{U}=\left\{(b_1,\dots,b_{|J|})\right\}$). In the sequel, we refer to $d$ as \emph{gamma factor}.

\begin{table}
	\centering
	\begin{tabular}{cl}
		\toprule
		$|I|$       & $10, 20, 30$ \\
		$p$         & $0.1, 0.2, 0.3$ \\
		$\left(k_1,k_2\right)$ & $(0,1), (10,10), (10, 50), (10, 100), (50, 50), (100, 100)$ \\ 
		$d$         & $0, 0.1, 0.2, 0.3, 0.4, 0.5, 0.6, 0.7, 0.8, 0.9, 1$ \\
		\bottomrule
	\end{tabular}
	\caption{Choices of the parameters for the random instances.}
	\label{table:random}
\end{table}

Table~\ref{table:random} summarizes all choices of the parameters. Thus, for $k:=(k_1,k_2)$, each combination of $|I|,p,k,d$ defines the structure of an instance for which we create $50$~representatives $I_1(|I|,p,k,d),\ldots, I_{50}(|I|,p,k,d)$. To be able to compare the impact of differing demand ranges or gamma factors, the underlying graph of instance $I_r(|I|,p,k,d)$ equals that of $I_r(|I|,p,k',d')$ for every $r$ and fixed values for $|I|$ and $p$. Furthermore, we are interested in the relative gap between the robust solution value and the worst case solution value, respectively average case solution value, to analyze the extra cost of robustness. For the random instances we choose ten extreme scenarios uniformly at random to compute the average case solution as explained above. Interestingly, the chosen median solution value of each instance~$I_r(|I|,p,k,d)$ is close to $\lceil\nicefrac{\Gamma}{3}\rceil$ so that almost every doctor covers three demand points. Thus, the average case solution value is always very close to the average over the trivial lower bounds $\lceil\nicefrac{\Gamma}{3}\rceil$.

In the following, we present some interesting findings during our analysis of these random instances. Fig.~\ref{fig:gaps} depicts the logarithmic (to the base of $10$) average relative gap between the objective values of the worst case solutions ($wcsv_r$) and the robust solutions ($ rsv_r$), i.~e.,~$\log_{10}\left(\nicefrac{1}{50}\sum_{r=1}^{50} \nicefrac{wcsv_r}{rsv_r}\right)$ (black), as well as the logarithmic average relative gap between the objective values of the average case solutions and the robust solutions (cyan). As a reference point the logarithmic average relative gap between the robust solution and itself is drawn as a dashed horizontal line (magenta). Moreover, the average processing times (in seconds) of \emph{asf} and \emph{setf} are displayed in Fig.~\ref{fig:random_times}, where we average only over the processing times for instances that were actually solved to optimality. The markers in Fig.~\ref{fig:random_times} are given an alpha value determining their transparency where the alpha is computed by dividing the number of solved instances by 50.

\begin{figure}%
\includegraphics[width=\columnwidth]{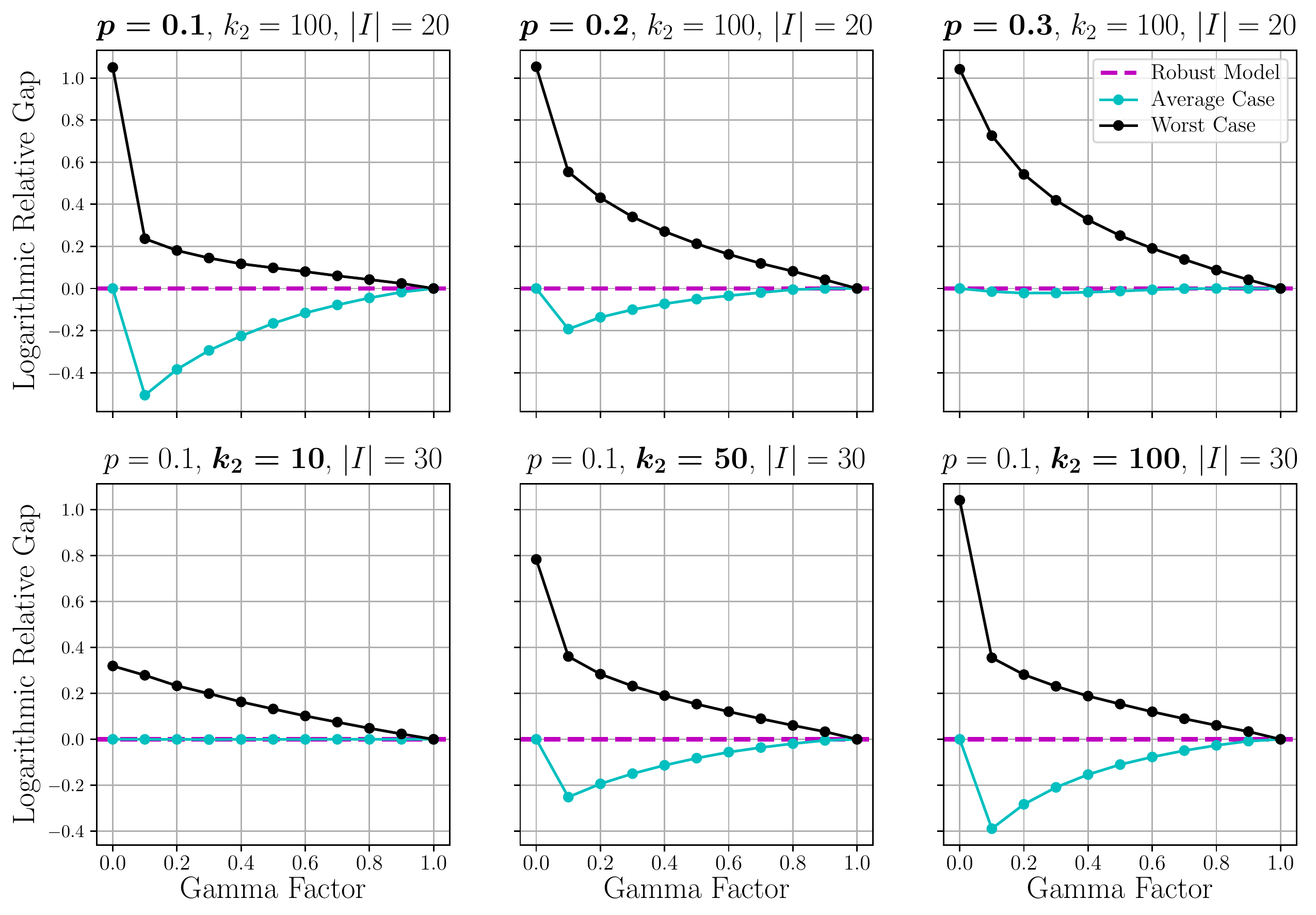}%
\caption{Logarithmic average relative gaps between the worst case and the robust solution as well as between the average case solution and the robust solution for $50$ random instances with $|J|=100$, $k_1=10$ and varying other parameters.}%
\label{fig:gaps}%
\end{figure}

Focusing on Fig.~\ref{fig:gaps} we first observe that for gamma factor $1.0$ or $0.0$ the objective values of the robust solution and the average case solution coincide as in these two cases there only exists one unique extreme scenario. Further, we can see that with rising gamma factor the objective value of the worst case solution gets closer to the objective value of the robust solution which is also expected.

Regarding the upper three plots in Fig.~\ref{fig:gaps} we can see that increasing the density of the graph decreases the relative gap between the average case solution values and the robust solution values
and increases the relative gap between the worst case solution values and the robust solution values. This can be explained by the fact that in dense graphs suppliers have more possibilities to serve clients
and can therefore act more \emph{flexible} than in sparse graphs. An extreme example would be a location adjacent to all regions. We could then simply put all suppliers in that location and get a
solution that is as high as the trivial lower bound. Furthermore, the relative gap between the worst case solution values and the robust solution values becomes larger with increasing density. Looking at the data we can see that the worst case solution and the average case solution do not change too much with increasing density, it is in fact the robust solution that becomes cheaper. We can conclude that in dense graphs
we get robustness almost for free, whereas in sparse graphs we have to pay quite a bit for turning the solution into a robust one. Note that increasing the number of locations while keeping $p$ fixed improves the robust solution in a similar manner.

In the second row of Fig.~\ref{fig:gaps} we see the impact of increasing the range for the upper demand of the regions. With increasing $k_2$ the relative gap between both the worst case solution values and the robust solution values as well as the average case solution values and the robust solution values becomes larger. This behavior does not solely depend on the increased value of $k_2$ but rather on the relative difference of $k_1$ and $k_2$. For example, on instances with $k_1=k_2=100$, $p=0.1$ and $|I|=30$ we can see that the robust solution values and the average solution values coincide in most cases. We conclude that the price of robustness is especially cheap if $k_1\geq k_2$ and the graph is \emph{dense} enough. A reason for this behavior could be that if all regions have high lower bounds on their demand (in comparison to their upper demand)  the number of suppliers positioned in locations adjacent to some fixed region is quite large in any average case solution. Therefore, there are more possibilities for allocating them to the clients.

Overall the price of robustness is very low if the graph is \emph{dense} or the relative gap between $k_1$ and $k_2$ is low. For sparse graphs the price of robustness can become very large. For example in the instances with $|I|=10$ and $p=0.1$ the robust solution value coincides with the worst case solution value quite often. Note that following the random graph construction given above, in all these instances every region is adjacent to exactly one location.

\begin{figure}%
\includegraphics[width=\columnwidth]{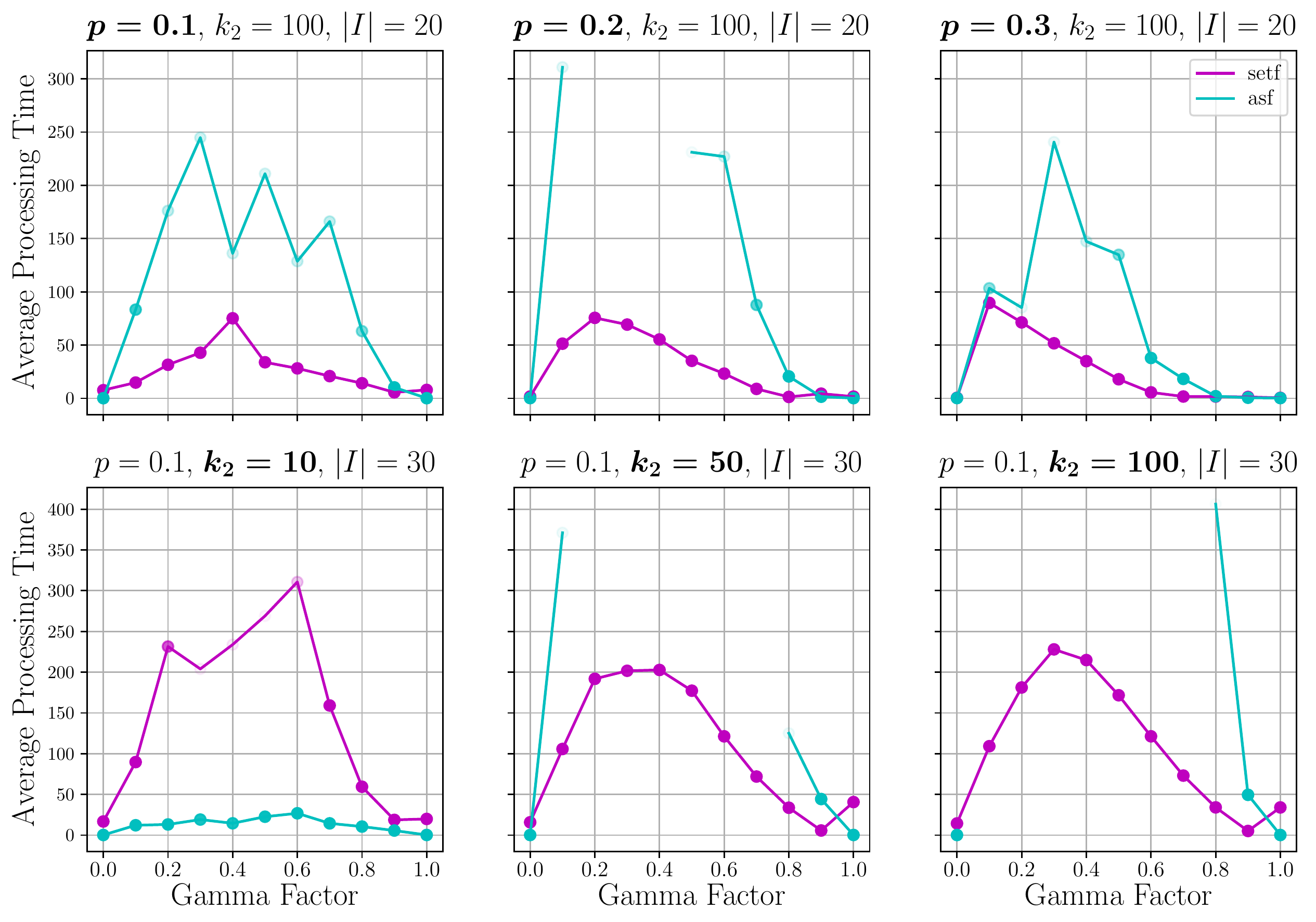}%
\caption{Average processing times in seconds of \emph{asf} and \emph{setf} for at most $50$ random instances with $|J|=100$, $k_1=10$ and
varying other parameters. The transparency of the markers reflect the amount of solved instances.}%
\label{fig:random_times}%
\end{figure}

Fig.~\ref{fig:random_times} depicts the average running times of the two solution approaches for  $50$ instances with $|J|=100$, $k_1=10$ and varying other parameters. If the plot does not 
contain a point for some gamma factor it means the corresponding solution approach did not finish computation in the given time window of two minutes. We can see, that in most cases \emph{setf} seems to be the better choice. Taking a closer look, it becomes clear that \emph{asf} performs especially well when the robust solution value coincides with the average case solution value. In these cases a very small number of extreme scenarios (often even just one) needs to be added in order to get a robust solution. With increasing number of required scenarios the running time for \emph{asf} explodes. The running time of \emph{setf} does not seem to depend on this too much. The only major impact on the running time of \emph{setf} seems to be the number of regions and locations. 

\subsection{Real world example: placing emergency doctors}
\label{sec:real_world}
As a real world application, we regard the problem of placing as few emergency doctors as possible into given facilities such that the emergencies happening in one shift can still be handled in a satisfactory way. In this context, the uncertain demand of each region reflects the unknown number of emergencies happening in that region during the considered shift. Thus, the proposed discrete budgeted uncertainty set allows, for each region, variations in a given interval but the total number of emergencies is bounded. The proposed model seems fitting for the application since all realistic scenarios should be covered equally well.\medskip

Using map data from OpenStreetMap~\cite{osm}, we construct a graph modeling the street network of some fixed part of the map.
In our computations, we considered the street network inside a bounding box enclosed in the federal state Rhineland-Palatinate in Germany. 
The size of the bounding box is roughly $300\text{km}^2$ mostly consisting of rural areas. The GPS coordinates of the south-west corner of the box are
$(7.2606, 49.1703)$ and the GPS coordinates of the north-east corner of the box are $(8.3890, 49.9537)$. Inside the box there are currently 
$38$ emergency facilities, cf.\ \cite{rettungswachen}, which we choose as locations of the instance denoted by the set $I$.

In the street network the edge weight corresponds to the time needed for a doctor to travel 
along this particular edge based on the maximum speed allowed on the associated piece of road.
To obtain the regions, for each street node, we compute a list of locations 
from which the street node can be reached within $15$~min. In 
Rhineland-Palatinate, $15$~min is the time at which the first responder must be present at the emergency 
after he left the facility. Now, we define all street nodes with the same list of 
facilities to be in the same region and denote the set of all regions by $J$. 
With our bounding box this results in a set of $426$ regions and gives a straight forward way to define the graph of the instance:
Simply add edges between each region and all locations of its list.

We further set $a_j=0$ and $b_j=1$ for all $j\in J$. The interpretation of these bounds is that in any 
region there might or might not occur an accident during the regarded shift.
We do not allow more than one emergency in a given region as our regions are rather small and
we therefore deem the case of more than one occurring emergency
during one shift to be unrealistic.
The total number of emergencies $\Gamma$ is then set to different values for the 
tests. 
We assume that an emergency doctor is able to handle up to three emergencies in one shift, 
i.e.,\ $q$ is fixed to three as in the random case above. 
Clearly, this assumption is not exact since emergencies may overlap and one doctor may not 
be able to attend to even two emergencies if they overlap.
Nevertheless, we are convinced that this approach is reasonable since adding uncertainty 
to the value $q$ as well would lead to overlapping uncertainties rapidly increasing the
conservativeness of the solution.

Note that our model does not forbid \emph{local worst cases}: 
For any subset of the regions of size at most $\Gamma$ there is a scenario in which 
the demand of each region in the subset is set to one. Thus, if the regarded part of
the map is too large, implicitly raising the assumed maximum number of emergencies $\Gamma$, the
solution to our model also covers scenarios in which emergencies massively occur in
very small parts of the entire map. Thus, when using \RobustOptShortT\ for this 
application, the regarded size of the map should be reasonable.\medskip

\begin{figure}%
\centering
\includegraphics[width=0.45\columnwidth]{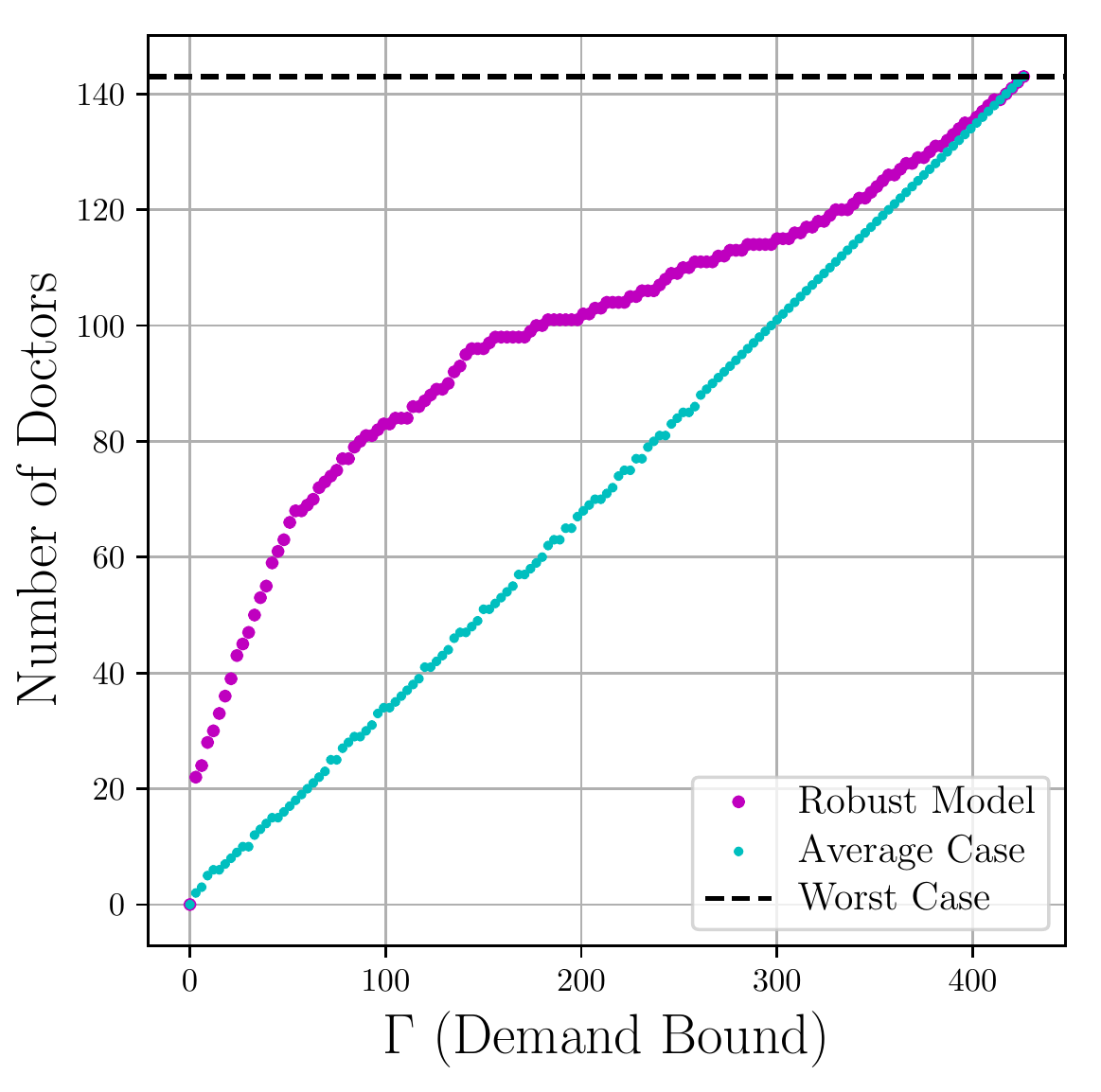}\quad\quad%
\includegraphics[width=0.45\columnwidth]{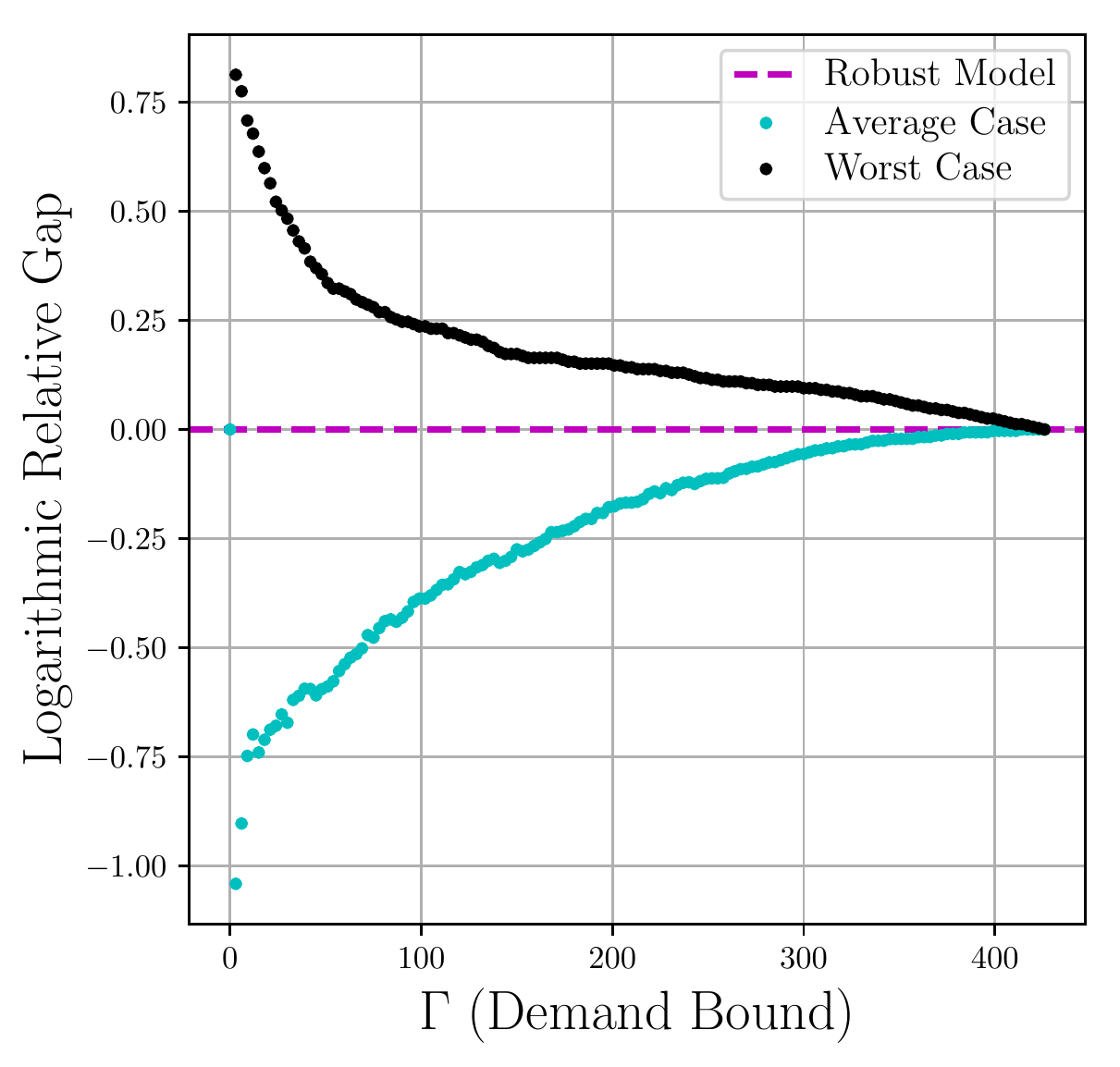}%
\caption{On the left, the number of doctors needed in the real world instances of the robust model, the average case and the worst case. On the right, the logarithmic (to base 10) relative gaps between the average case solution and the robust solution as well as between the worst case solution and the robust solution.}%
\label{fig:robust_vs_worst}%
\end{figure}

Fig.~\ref{fig:robust_vs_worst} shows the comparison of the objective value of the robust model to the 
median objective value of the average case. Due to larger input data compared to the random instances, we 
choose five random extreme scenarios for the average case solution. Fig.~\ref{fig:time_robust} depicts the running times of \emph{setf} and \emph{asf}.
Test runs were made for each ${\Gamma\in\{3i\colon i=0,\dots,142\}}$.

Taking a closer look at Fig.~\ref{fig:robust_vs_worst}, we can see that 
the objective value of the average case solution is always close to the trivial lower bound $\left\lceil\nicefrac{\Gamma}{3}\right\rceil$.
Thus, almost all doctors cover three emergencies in the fixed 
average scenarios. 
The behavior of the robust solution is somehow expected. The greatest absolute deviation from the average 
solution is attained for $\Gamma$ between~$54$ and~$156$, so roughly for $\Gamma$ attaining a 
value between $\nicefrac{1}{10}$ and $\nicefrac{1}{40}$ of the total sum of the demands~$426$. 
The relative distance between the average case solution value and the robust solution value decreases linearly. 
In the regarded map area, fixing the number of emergencies between $20$ and $40$ seems adequate. 
The price of robustness for the application in this area does seem quite low, given the fact that in the average case solutions
not even all regions have to be reachable by some doctor.

The running time of \emph{setf} is acceptable for all regarded demand bounds $\Gamma$. On the other hand, \emph{asf} was not able to 
solve all instances within the time limit of $15$~min. Furthermore, the variance in the running time is much higher than for \emph{setf}. 
Interestingly, in some cases \emph{asf} outperforms \emph{setf} significantly. Especially in the area where $\Gamma$ is around $150$. Thus, it 
seems worthwhile to use both approaches in practice parallelly. If one is looking for only one solution approach \emph{setf} should be 
preferred because it seems more reliable
as the variance in the running times is smaller.
Surprisingly, for random instances created as described in the previous section with parameters similar to this application ($|I|=38$, $|J|=426$, $a_j=0$, $b_j=1$, $p=0.07$), the running times of both solution approaches increase rapidly, where the given $p$ roughly models the density of the graph arising from our application. For example, \emph{setf} took roughly $20$~h of computation for solving instances with $\Gamma=45$. For larger $\Gamma$, it did not finish solving the problem after $48$~h. \emph{asf} was not able to finish any instance within the time limit of 48~hours. We think that 
this behavior is due to the \emph{planar-like} structure of the graph arising from
the application. For future research, it might be interesting to 
work on complexity results for instances with these planar-like structures.

We conclude that for map areas of the tested size 
\RobustOptShortT\ is of interest for the given practical application. Though, we think that for a 
larger regarded map area, a model including 
some local condition on the scenarios to prevent the described massive occurrences of emergencies
in comparably small map areas might suit the application even better. This could be a direction
of future research on the topic.

\begin{figure}%
\centering
\includegraphics[width=\columnwidth]{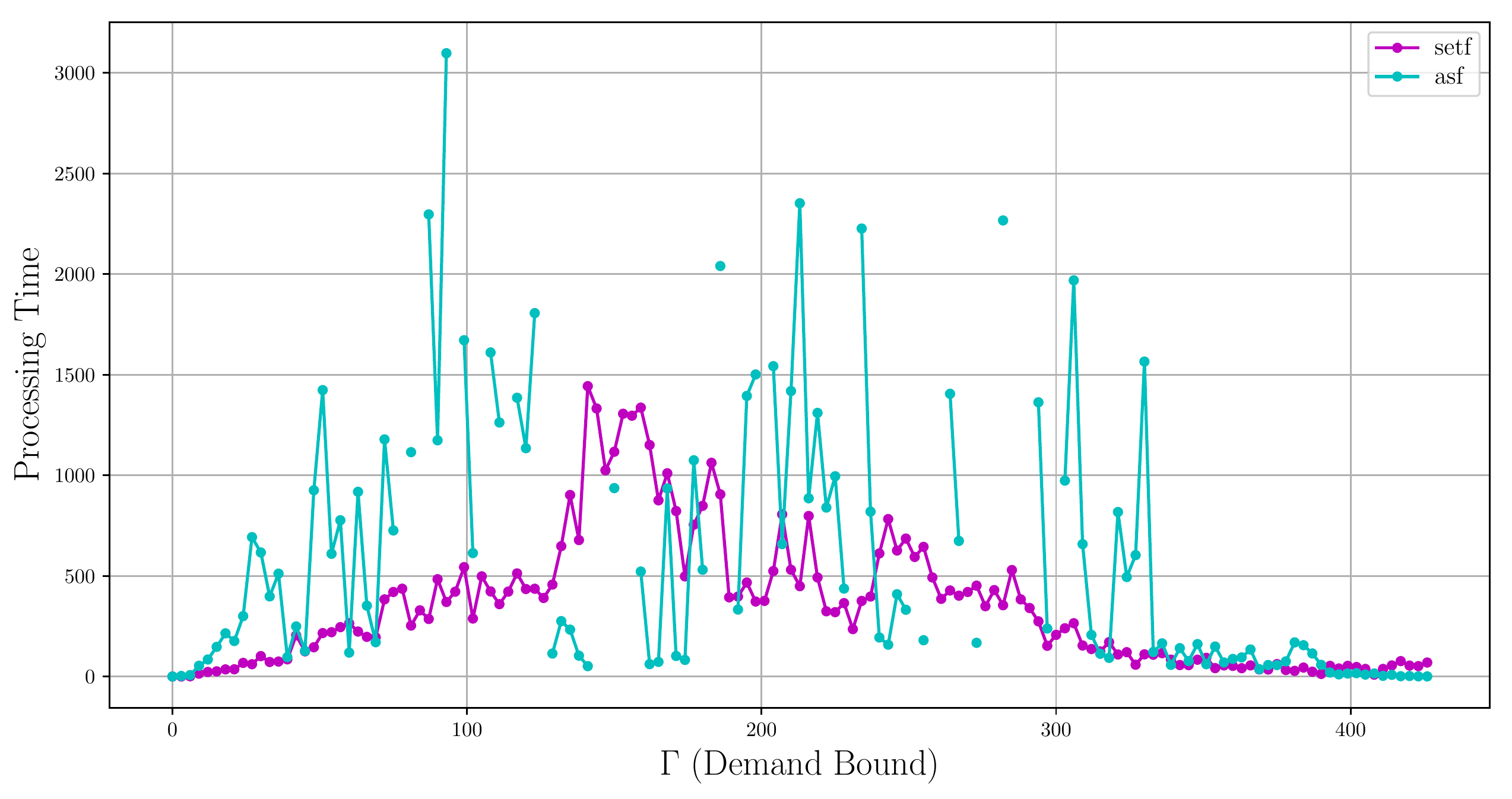}%
\caption{Processing times in seconds of \emph{setf} and \emph{asf} for different demand bounds.}%
\label{fig:time_robust}%
\end{figure}



\section{Conclusion}
\label{sec: Conclusion}
We have presented a novel problem called \nonRobust\ which we have identified to be a special case of \emph{Multiset Multicover}.  
We have shown that it is NP-complete for fixed values of~$q\in\mathbb{N}$ with $q\geq 3$ but polynomial time solvable for $q=1,2$. 
The main focus of this paper was the robust version of \nonRobust, which we proved to be strongly NP-hard for all $q\in\mathbb{N}$.
Further we have given two different integer programming formulations of the problem and discussed their up- and downsides. 
We presented strategies for solving the problems based on constraint generation.  
Our computational results based on random instances and instances corresponding to a real world application are quite promising.  They show that the model and the robust approach can be of great use for practical problems since it is able to hedge against uncertainty with fewer resources compared to an all worst-case approach.


\section{Acknowledgments}
\label{sec: Acknowledgements}
Map data copyrighted OpenStreetMap contributors and available from \\
\texttt{https://www.openstreetmap.org}.

\noindent This work was partially supported by the German Federal Ministry of Education and Research within the project ''HealthFaCT - Health: Facility Location, Covering and Transport''.

\noindent We thank the anonymous referees for their careful reading of our manuscript and their many insightful comments and suggestions.

\bibliographystyle{plain}

\clearpage
\appendix

\section{Complexity Results for \nonRobustShort}
\label{appendix: complexity}

\noindent At this point we will give formal proofs for the
statements concerning the complexity of \nonRobustOptShort\
in Section~\ref{sec: Problem Definition and Classifications}.

\begin{observation}
	\emph{Min-$1$-Multiset Multicover} is solvable in linear time.
\end{observation}

\begin{proof}
	If $q=1$, given an instance for \nonRobustOptShort, in any solution
	each client needs to be assigned a unique supplier.  This means, for
	each client in some region $j\in J$ we may put a single supplier in
	some location $i\in N(j)$.  This will yield a feasible solution with
	$\sum_{i\in I}d_i$ suppliers, which is clearly optimal.  We can find
	this solution in time linear in~$|I|+|J|$.
\end{proof}

For $q=2$ we can still solve \nonRobustOptShort\ in polynomial time.
We show how to compute an optimal solution using an algorithm
for the \emph{Edge Cover} problem.  Recall that for a graph $G=(V, E)$
an \emph{edge cover} is a subset of the edges $E^\prime\subseteq E$,
such that each vertex $v\in V$ is incident to at least one edge
$e\in E^\prime$.  Regard the following procedure:  For a given
instance of the problem, duplicate each region $j\in J$ exactly $d_j$
times yielding a set $V_j$ for each $j\in J$.  We see that we may
bound any $d_j$ by~$|I|$ in the proof of Theorem~\ref{thm:q2}. Regard the graph $G=(V,E)$ with vertex set
$V=\bigcup_{j\in J} V_j$ where the edge $(u,v)$ for $u\in V_{j_1}$,
$v\in V_{j_2}$ are in $E$ if
$N(j_1)\cap N(j_2)\neq\emptyset$. Note that this implies that the 
graph induced by some set $V_j$ is the complete graph.
Next compute a minimum edge cover
$E'$ in $G$ and initially set $x_i=0$ for all $i\in I$.  
For each edge $(u,v)\in E'$, with $u\in V_{j_1}$,
$v\in V_{j_2}$ we increase $x_i$ by $1$ for some 
$i \in N(j_1)\cap N(j_2)$, meaning we add a supplier in location $i$ who covers one demand point in region~$j_1$ and one in region~$j_2$.

\begin{theorem}
	\label{app:thm:q2}
	The above procedure solves \emph{Min-$2$-Multiset Multicover} and
	can be implemented to run in time
	$O(|I|^{5/2}|J|^{5/2})$.
\end{theorem}

\begin{proof}
	We first prove the correctness of the procedure.  Let $G=(V,E)$ be
	the graph defined in the procedure.  Let $x$ be as defined by the
	procedure above and let $E^\prime$ be the minimum
	edge cover from the procedure.  As we have a node in $G$ for every
	client and the nodes corresponding to the clients are covered by the
	edges in $E^\prime$ it is clear that $x$ defines a feasible solution
	for \nonRobustOptShort.  It remains to show, that given a solution
	$x$ to \nonRobustOptShort, there is an edge cover with
	$\sum_{i\in I}x_i$ edges.  By the equivalence of (IP
	\ref{ip:subset}) and (MIP~\ref{ip:allocation}) we can find
	$y_{ij}\in\mathbb{Z}_{\geq 0}$ for all $i\in I$, $j\in J$ fulfilling
	\begin{align*}
	\sum_{i\in N(j)}y_{ij}\geq d_j~\forall j\in~J\text{ and }\sum_{j\in N(i)}y_{ij}\leq 2x_i~\forall i\in~I.
	\end{align*}
	Clearly, we may assume equality in the second set of equations and can
	thereby determine for each supplier the two clients he serves. We initially set 
	$E^\prime$ to the empty set. If,
	for each supplier, we now select the edge between the two clients he
	serves and add it to $E^\prime$, we get an edge cover of $G$ with
	$\sum_{i\in I}x_i$ edges.  This proves the correctness of the
	procedure.
	
	To see the running time, first note that we may bound the number of
	clients $d_j$ in any region by the number of suppliers: Assume
	$d_j\geq |I|+1$ for some $j\in J$ in some instance.  Then, in any
	solution of (MIP~\ref{ip:allocation}) there is some $i\in N(j)$ such
	that $y_{ij}\geq 2$.  
	Thus, given an optimal solution, choose $i$ such that
	$y_{ij}\geq 2$. Removing one supplier from $i$ now yields an optimal 
	solution to the same instance with the demand of region $j$ being $d_j - 2$.
	As a consequence we may also solve this instance and then afterwards add 
	a supplier to any location connected to $j$ to get an optimal solution
	of the original problem. We can therefore decrease the demands of all $j$
	with $d_j\geq |I|+1$ to $|I|$, respectively
	$|I|-1$ by adding $\left\lceil\nicefrac{1}{2}\left(d_j - |I|\right)\right\rceil$ doctors to any location connected to $j$. This can be done
	in constant time for any region $j\in J$. With this observation, it can
	readily be seen that the constructed graph has at most $N:=|I|\cdot |J|$
	vertices whereas the number $M$ of edges is upper bounded by
	$O(|I|^2 |J|^2)$. A minimum edge cover in a graph with $N$~vertices and
	$M$~edges can be obtained by first solving a maximum matching problem
	in time $O(\sqrt{N}M\log_N (N^2/M))$~\cite{Goldberg+Karzanov:flow} and
	then using $O(M)$ time to augment the
	matching~\cite{Schrijver:book:new,Lovasz:matching,Garey}.   This gives the claimed running
	time.
\end{proof}

It is fairly easy to see that \nonRobust\ is a generalization of
\emph{Set Cover by $q$-sets}, i.e., the restriction of \emph{Set
	Cover} where all sets are of size exactly~$q$.  Since this is an
NP-complete problem, the next result is not surprising.  For the sake
of completeness we will nevertheless give a formal proof.

\begin{theorem}
	\label{app:the:msmqc_npcomplete}
	For any fixed $q\geq 3$, \nonRobust\ is NP-complete in the strong sense.
\end{theorem}

\begin{proof}
	Let $q\in\mathbb{N}$ with $q\geq 3$.  As a consequence of Lemma
	\ref{lemma:ipequivalence}, for a given instance of \nonRobustShort,
	we may test if a given solution $x$ is feasible by one Max-Flow
	computation.  Therefore, \nonRobustShort\ is contained in~NP.
	
	To see that the problem is NP-hard in the strong sense we illustrate
	a reduction from \emph{Exact Cover by $3$-sets}, which is known to
	be NP-hard in the strong sense, cf. \cite{Garey}.
	Let $X$ be a set and $\mathcal{S}$ be a collection of subsets of $X$
	where $|S|=3$ for all $S\in\mathcal{S}$.  We create an instance of
	\nonRobustShort\ in the following way.  Due to legibility, assume
	the subsets $S\in\mathcal{S}$ have unique indices $i_S$.  Let
	$I:=~\{i_S\colon~S\in\mathcal{S}\}$, $J:=X$ and define the graph of
	the instance via $N(i_S)=S$ for all $S\in\mathcal{S}$.  Further, let
	$d_j=1$ for all $j\in X$ and $B=\nicefrac{|X|}{3}$.  Now, let
	$\mathcal{S}'$ be a solution to the instance of \emph{Exact Cover by
		$3$-sets}.  Clearly, setting $x_S$ to one if and only if
	$S\in\mathcal{S}'$ and zero else yields a feasible solution to
	\nonRobustShort\ with $\sum_{S\in \mathcal{S}}x_{i_S}=B$.  On the other hand, note
	that in any solution $x$ to \nonRobustShort\ $x_{i_S}\leq 1$ for all
	$S\in \mathcal{S}$.  Furthermore, since $|N(i_S)|=3$ the actual value of $q$
	is of no further interest as long as $q\geq 3$.  Thus,
	$\mathcal{S}'=\{S\colon x_S > 0\}$ is a solution to \emph{Exact
		Cover by $3$-sets}.
\end{proof}

Remark~\ref{rem:multiset_multicover} reveals \nonRobustShort\
to be a special case of
\emph{Multiset Multicover}.
It is well known that \emph{Multiset
	Multicover} can be approximated within a factor of $\log(s)$ where
$s$ is the size of the largest multiset of an instance, see
e.g.~\cite{Hochbaum:book,Lovasz:cover,Chvatal:cover}. If we regard
\nonRobustOptShort\ as \emph{Multiset Multicover} problem, all
multisets have fixed size~$q$.  We therefore automatically get a
$\log(q)$ approximation for \nonRobustOptShort:

\begin{observation}
	There is a $\log(q)$ approximation for \nonRobustOpt.
\end{observation}

\end{document}